\documentclass[pdflatex,sn-mathphys-num]{sn-jnl}

\usepackage{graphicx}%
\usepackage{multirow}%
\usepackage{amsmath,amssymb,amsfonts}%
\usepackage{amsthm}%
\usepackage{mathrsfs}%
\usepackage[title]{appendix}%
\usepackage{textcomp}%
\usepackage{manyfoot}%
\usepackage{booktabs}%
\usepackage{algorithm}%
\usepackage{algorithmicx}%
\usepackage{algpseudocode}%
\usepackage{listings}%

\usepackage{siunitx}

\usepackage{amsmath,amsthm,amsfonts,amssymb,bm}
\usepackage{mathtools}
\usepackage{enumitem}\setitemize{leftmargin=5mm}
\usepackage{accents}
\usepackage{color}
\usepackage{booktabs}
\usepackage{subcaption}
\usepackage[usenames,dvipsnames,svgnames,table]{xcolor}
\usepackage[capitalize]{cleveref}

\usepackage{csvsimple}

\usepackage{tikz}
\usetikzlibrary{calc}

\usepackage{todonotes}
\usepackage{pgfplotstable}

\theoremstyle{thmstyleone}%
\newtheorem{theorem}{Theorem}
\newtheorem{lemma}[theorem]{Lemma}

\newtheorem{assumption}[theorem]{Assumption}

\theoremstyle{thmstyletwo}%
\newtheorem{remark}{Remark}%

\theoremstyle{thmstylethree}%
\newtheorem{definition}{Definition}%

\usepackage{verbatim}

\usepackage{stmaryrd} %
\usepackage{pgf,tikz}
\usetikzlibrary{spy,matrix, calc, arrows,shapes,decorations}
\usepackage{tkz-euclide}
\usetikzlibrary{positioning,arrows,calc,decorations.markings,math,arrows.meta}
\usepackage{pgfplots} 
\usepgfplotslibrary{groupplots}
\usepackage{csvsimple} 
\usepackage{siunitx} %
\usepackage{tabularx}
\usepackage{booktabs}
\usepackage{multirow}
\usepackage[normalem]{ulem}
\useunder{\uline}{\ul}{}

\pgfplotscreateplotcyclelist{noMarkPlot}{%
 teal,every mark/.append style={solid,fill=teal!80!black},mark=*,mark size=0pt\\
 red!70!white,every mark/.append style={solid,fill=red!80!black},mark=pentagon*,mark size=0pt\\
 ForestGreen,every mark/.append style={solid,fill=ForestGreen!80!black},mark=triangle*,mark size=0pt\\
}

\pgfplotscreateplotcyclelist{barcolors}{%
teal,fill=teal,fill opacity=0.7,thick,postaction={pattern=crosshatch dots,pattern color=teal}\\
red!70!white,fill=red!70!white,fill opacity=0.7,thick,postaction={pattern=crosshatch dots, pattern color=red!70!white}\\
ForestGreen,fill=ForestGreen,fill opacity=0.7,thick,postaction={pattern=crosshatch dots, pattern color=ForestGreen}\\
cyan!60!black,fill=cyan!60!black,fill opacity=0.7,thick,postaction={pattern=crosshatch dots, pattern color=cyan!60!black}\\
orange,fill=orange,fill opacity = 0.7,thick,postaction={pattern=crosshatch dots, pattern color=orange}\\
}

\pgfplotscreateplotcyclelist{MachNum}{%
 teal,every mark/.append style={solid,fill=teal!80!black},mark=*,mark size=2.5pt\\
 teal,every mark/.append style={solid,fill=teal!80!black},mark=*,mark size=2.5pt\\
 ForestGreen,every mark/.append style={solid,fill=ForestGreen!80!black},mark=triangle*,mark size=2.5pt\\
 ForestGreen,every mark/.append style={solid,fill=ForestGreen!80!black},mark=triangle*,mark size=2.5pt\\
}

\pgfplotscreateplotcyclelist{LiftvsSIP}{%
 teal,every mark/.append style={solid,fill=teal!80!black},mark=*,mark size=2.5pt,opacity=.8\\
 orange,every mark/.append style={solid,fill=orange!80!black},mark=square*,mark size=2.5pt,opacity=.8\\
 ForestGreen,every mark/.append style={solid,fill=ForestGreen!80!black},mark=triangle*,mark size=2.5pt,opacity=.8\\
 red!70!white,every mark/.append style={solid,fill=red!80!black},mark=pentagon*,mark size=2.5pt\\
}

\pgfplotscreateplotcyclelist{ConvPlot}{%
 teal,every mark/.append style={solid,fill=teal!80!black},mark=*,mark size=2.5pt\\
 red!70!white,every mark/.append style={solid,fill=red!80!black},mark=pentagon*,mark size=2.5pt\\
 ForestGreen,every mark/.append style={solid,fill=ForestGreen!80!black},mark=triangle*,mark size=2.5pt\\
 cyan!60!black,every mark/.append style={solid,fill=cyan!80!black},mark=diamond*,mark size=2.5pt\\
}

\pgfplotscreateplotcyclelist{unfmixed1}{%
 teal,every mark/.append style={solid,fill=teal!80!black},mark=*,mark size=2.5pt\\
 red!70!white,every mark/.append style={solid,fill=red!80!black},mark=pentagon*,mark size=2.5pt\\
 ForestGreen,every mark/.append style={solid,fill=ForestGreen!80!black},mark=triangle*,mark size=2.5pt\\
cyan!60!black,every mark/.append style={solid,fill=cyan!80!black},mark=diamond*,mark size=2.5pt\\
orange,every mark/.append style={solid,fill=orange!80!black},mark=square*,mark size=2.5pt\\
}

\pgfplotscreateplotcyclelist{unfmixed2}{%
teal,every mark/.append style={solid,fill=teal!80!black},mark=*,mark size=2.5pt\\
orange,every mark/.append style={solid,fill=orange!80!black},mark=square*,mark size=2.5pt\\
}

\pgfplotscreateplotcyclelist{unfmixed3}{%
 orange,every mark/.append style={solid,fill=orange!80!black},mark=square*,mark size=2.5pt\\
 teal,every mark/.append style={solid,fill=teal!80!black},mark=*,mark size=2.5pt\\
 cyan!60!black,every mark/.append style={solid,fill=cyan!80!black},mark=diamond*,mark size=2.5pt\\
 red!70!white,every mark/.append style={solid,fill=red!80!black},mark=pentagon*,mark size=2.5pt\\
 ForestGreen,every mark/.append style={solid,fill=ForestGreen!80!black},mark=triangle*,mark size=2.5pt\\
}

\pgfplotscreateplotcyclelist{paulcolors}{%
violet,every mark/.append style={solid,fill=violet},mark=square*,very thick,mark size=3pt\\%
teal,every mark/.append style={solid,fill=teal},mark=*,very thick,mark size=3pt\\%
orange,every mark/.append style={solid,fill=orange},mark=diamond*,very thick,mark size=3pt\\%
violet,densely dashed,every mark/.append style={solid,fill=violet},mark=square*,very thick,mark size=3pt\\%
teal,densely dashed,every mark/.append style={solid,fill=teal},mark=*,very thick,mark size=3pt\\%
orange,densely dashed,every mark/.append style={solid,fill=orange},mark=diamond*,very thick,mark size=3pt\\%
violet,dash dot,every mark/.append style={solid,fill=violet},mark=square*,very thick,mark size=3pt\\%
teal,dash dot,every mark/.append style={solid,fill=teal},mark=*,very thick,mark size=3pt\\%
orange,dash dot,every mark/.append style={solid,fill=orange},mark=diamond*,very thick,mark size=3pt\\%
}
\pgfplotscreateplotcyclelist{paulcolors2}{%
teal,every mark/.append style={solid,fill=teal},mark=*,very thick,mark size=3pt\\%
orange,every mark/.append style={solid,fill=orange},mark=diamond*,very thick,mark size=3pt\\%
teal,densely dashed,every mark/.append style={solid,fill=teal},mark=*,very thick,mark size=3pt\\%
orange,densely dashed,every mark/.append style={solid,fill=orange},mark=diamond*,very thick,mark size=3pt\\%
teal,dash dot,every mark/.append style={solid,fill=teal},mark=*,very thick,mark size=3pt\\%
orange,dash dot,every mark/.append style={solid,fill=orange},mark=diamond*,very thick,mark size=3pt\\%
}
\pgfplotscreateplotcyclelist{paulcolors1}{%
teal,every mark/.append style={solid,fill=teal},mark=*,very thick,mark size=3pt\\%
}
\pgfplotscreateplotcyclelist{paulcolorsfill}{%
teal,fill=teal,fill opacity=0.5,thick\\
orange,fill=orange,fill opacity=0.5,thick\\
violet,fill=violet,fill opacity=0.5,thick\\
}

\pgfplotsset{
    discard if not/.style 2 args={
        x filter/.append code={
            \edef\tempa{\thisrow{#1}}
            \edef\tempb{#2}
            \ifx\tempa\tempb
            \else
                
            \fi
        }
    }
}
\pgfplotsset{compat=1.16}%

\pgfplotsset{tick label style={font=\small},label style={font=\small},legend style={font=\small},}
\pgfplotsset{ width=.49\linewidth}

\usepackage{listings}

\definecolor{maroon}{cmyk}{0, 0.87, 0.68, 0.32}
\definecolor{halfgray}{gray}{0.55}
\definecolor{ipython_frame}{RGB}{207, 207, 207}
\definecolor{ipython_bg}{RGB}{247, 247, 247}
\definecolor{ipython_red}{RGB}{186, 33, 33}
\definecolor{ipython_green}{RGB}{0, 128, 0}
\definecolor{ipython_cyan}{RGB}{64, 128, 128}
\definecolor{ipython_purple}{RGB}{170, 34, 255}

\lstdefinelanguage{iPython}{
    morekeywords=[2]{abs,all,any,basestring,bin,bool,bytearray,callable,chr,classmethod,cmp,compile,complex,delattr,dict,dir,divmod,enumerate,eval,execfile,file,filter,float,format,frozenset,getattr,globals,hasattr,hash,help,hex,id,input,int,isinstance,issubclass,iter,len,list,locals,long,map,max,memoryview,min,next,object,oct,open,ord,pow,property,range,raw_input,reduce,reload,repr,reversed,round,set,setattr,slice,sorted,staticmethod,str,sum,super,tuple,type,unichr,unicode,vars,xrange,zip,apply,buffer,coerce,intern},%
    sensitive=true,%
    morecomment=[l]\#,%
    morestring=[b]',%
    morestring=[b]",%
    morestring=[s]{'''}{'''},%
    morestring=[s]{"""}{"""},%
    morestring=[s]{r'}{'},%
    morestring=[s]{r"}{"},%
    morestring=[s]{r'''}{'''},%
    morestring=[s]{r"""}{"""},%
    morestring=[s]{u'}{'},%
    morestring=[s]{u"}{"},%
    morestring=[s]{u'''}{'''},%
    morestring=[s]{u"""}{"""},%
    literate=
    {á}{{\'a}}1 {é}{{\'e}}1 {í}{{\'i}}1 {ó}{{\'o}}1 {ú}{{\'u}}1
    {Á}{{\'A}}1 {É}{{\'E}}1 {Í}{{\'I}}1 {Ó}{{\'O}}1 {Ú}{{\'U}}1
    {à}{{\`a}}1 {è}{{\`e}}1 {ì}{{\`i}}1 {ò}{{\`o}}1 {ù}{{\`u}}1
    {À}{{\`A}}1 {È}{{\'E}}1 {Ì}{{\`I}}1 {Ò}{{\`O}}1 {Ù}{{\`U}}1
    {ä}{{\"a}}1 {ë}{{\"e}}1 {ï}{{\"i}}1 {ö}{{\"o}}1 {ü}{{\"u}}1
    {Ä}{{\"A}}1 {Ë}{{\"E}}1 {Ï}{{\"I}}1 {Ö}{{\"O}}1 {Ü}{{\"U}}1
    {â}{{\^a}}1 {ê}{{\^e}}1 {î}{{\^i}}1 {ô}{{\^o}}1 {û}{{\^u}}1
    {Â}{{\^A}}1 {Ê}{{\^E}}1 {Î}{{\^I}}1 {Ô}{{\^O}}1 {Û}{{\^U}}1
    {œ}{{\oe}}1 {Œ}{{\OE}}1 {æ}{{\ae}}1 {Æ}{{\AE}}1 {ß}{{\ss}}1
    {ç}{{\c c}}1 {Ç}{{\c C}}1 {ø}{{\o}}1 {å}{{\r a}}1 {Å}{{\r A}}1
    {€}{{\EUR}}1 {£}{{\pounds}}1
    {^}{{{\color{ipython_purple}\^{}}}}1
    {=}{{{\color{ipython_purple}=}}}1
    {+}{{{\color{ipython_purple}+}}}1
    {*}{{{\color{ipython_purple}$^\ast$}}}1
    {/}{{{\color{ipython_purple}/}}}1
    {+=}{{{+=}}}1
    {-=}{{{-=}}}1
    {*=}{{{$^\ast$=}}}1
    {/=}{{{/=}}}1,
    literate=
    *{-}{{{\color{ipython_purple}-}}}1
     {?}{{{\color{ipython_purple}?}}}1,
    identifierstyle=\color{black}\ttfamily,
    commentstyle=\color{ipython_cyan}\ttfamily,
    stringstyle=\color{ipython_red}\ttfamily,
    keepspaces=true,
    showspaces=false,
    showstringspaces=false,
    rulecolor=\color{ipython_frame},
    framexleftmargin=0mm,
    numbers=left,
    numberstyle=\tiny\color{halfgray},
    numbersep=1mm,
    xleftmargin=1mm,
    basicstyle=\scriptsize,
    keywordstyle=\color{ipython_green}\ttfamily,
}

\lstdefinestyle{trefftzy}{
    language=iPython,
    emptylines=1,
    breaklines=true,
    basicstyle=\footnotesize\ttfamily\color{black},    
    moredelim=**[is][\color{teal}]{<}{>},
    moredelim=**[is][\color{purple}]{'}{'},
}

\newcommand{\nv}{\boldsymbol{\nu}}

\DeclareMathOperator{\Id}{\mathrm{id}}

\newcommand{\Th}{{\mathcal{T}_h}}

\newcommand{\Fh}{\mathcal{F}_h}

\newcommand{\jump}[1]{[\![ #1 ]\!]}

\newcommand{\hdgjump}[1]{[\![ \underline{#1} ]\!]}

\newcommand{\seqh}[1]{(#1_h)_{h \in \mathcal{H}}}

\renewcommand\Re{\operatorname{Re}}
\renewcommand\Im{\operatorname{Im}}

\newcommand{\lsp}{\Pi^{p-2}_{F}}
\newcommand{\lspPerp}{(\Id-\lsp)}
\newcommand{\hessian}{D^2}
\let \vec \mathbf

\newcommand{\stab}{\eta}

\newcommand{\sterm}{\mathcal{S}}

\newcommand{\dn}{\partial_{\bm{\nu}}}
\newcommand{\nHn}[1]{\vec{n}^\top \! (\hessian #1) \vec{n}}

\renewcommand{\ul}[1]{\underline{#1}}

\renewcommand{\u}{\ul{u}}
\newcommand{\pih}{\ul{\pi}_h}
\newcommand{\Pih}{\ul{\Pi}_h}

\newcommand{\uh}{\ul{u}_h}
\newcommand{\uvol}{u_h}
\newcommand{\ufac}{\sigma_{F}}

\newcommand{\vh}{\ul{v}_h}
\newcommand{\vvol}{v_h}
\newcommand{\vfac}{\mu_{F}}

\newcommand{\wh}{\ul{w}_h}
\newcommand{\wvol}{w_h}

\newcommand{\GammaR}{\Gamma_{\! \! R}}
\newcommand{\GammaD}{\Gamma_{\! \! D}}

\newcommand{\lift}{\mathcal{R}_h}

\newcommand{\Am}{\mathbb{A}}
\newcommand{\FOtermu}{\alpha \Delta u + \beta \nHn{u}}

\renewcommand{\FOtermu}{\mathbb{A} \! : \! \hessian u}

\newcommand{\FOtermuh}{\alpha \Delta \uvol + \beta \nHn{\uvol}}
\renewcommand{\FOtermuh}{\mathbb{A} \! : \! \hessian \uvol}

\newcommand{\FOterm}[1]{\mathbb{A} \! : \! \hessian #1}

\newcommand{\frobip}[2]{#1 \! : \! #2}

\newcommand{\CR}{C_{\! R}}

\newcommand{\CE}{C_{\! E}}

\newcommand{\dCr}{\delta_{\CR}}
\newcommand{\dCe}{\delta_{C_{\! E}}}

\newcommand{\Clift}{C_{\lift}}

\newcommand{\sconst}{C_{\text{stab}}}
\newcommand{\Idm}{\mathbb{I}}

\newcommand{\Pvol}{\Pi_h^{\text{vol}}}

\begin{document}
\title{A hybrid $C^{0}$-interior penalty method for the nematic Helmholtz--Korteweg equation}

\author[1]{\fnm{Tim} \sur{van Beeck}}\email{t.beeck@math.uni-goettingen.de}

\affil[1]{\orgdiv{Institute for Numerical and Applied Mathematics}, \orgname{ University of G\"{o}ttingen}, \orgaddress{\street{Lotzestr. 16-18}, \city{G\"{o}ttingen}, \postcode{37083}, \country{Germany}}}

\keywords{Nematoacoustics, Hybridization, $C^0$-interior penalty, Compact perturbations, Cordes condition}

\abstract{The nematic Helmholtz--Korteweg equation models the propagation of time-harmonic acoustic waves in nematic Korteweg fluids, such as nematic liquid crystals. The PDE augments the classical Helmholtz equation with two additional fourth-order terms, one of which is anisotropic in the direction of the nematic field. We refine the previous continuous analysis of Farrell et al.~(2025) by using the Cordes condition and present a $C^0$-hybrid interior penalty discretization. The proposed discretization offers greater flexibility than $C^1$-conforming methods and is well-suited for applications in three dimensions and on curved domains. We prove stability of the method for any polynomial degree greater than or equal to two, independent of the spatial dimension, provided that the anisotropy is sufficiently small. Further, we show that the sequence of discrete solutions converges to the continuous solution under minimal regularity assumptions and derive convergence rates if the continuous solution has additional regularity. Finally, we illustrate the capabilities of the method through numerical examples.}

\maketitle

\section{Introduction}
\noindent The mathematical and physical communities have studied liquid crystals intensively in the last decades \cite{S16,B17,V18,WZZ21}, not only because of the intriguing challenges they present but also due to their significant potential for technological applications. One prominent example is Liquid Crystal Displays (LCDs), whose core mechanism is based on the principle that polarized light waves can be manipulated by controlling the orientation of the liquid crystal molecules.

In this work, we consider the propagation of time-harmonic acoustic waves in \emph{nematic liquid crystals}, which are liquid crystals whose local average orientation is described by a unit vector $\bm{n} \in \mathbb{R}^d$, known as the nematic director. As in LCDs, the nematic director interacts with the propagating wave and can be controlled by an external electromagnetic field, thereby enabling control of the propagation itself. In other words, the nematic director induces \emph{anisotropy} in the attenuation and the speed of sound of acoustic waves; in particular, the wave propagates faster in the direction of the nematic director $\bm{n}$ \cite{LL70,SelingerEtAl02,SV12,V18,FZ24,FvBZ25}. We illustrate this with an example in \Cref{fig:intro:waveguide}. The study of such interactions is known as \emph{nematoacoustics} \cite[Chap.~5]{SV12} \cite{V09}.

The propagation of time-harmonic acoustic waves in nematic liquid crystals can be modeled by the \emph{nematic Helmholtz--Korteweg equation} \cite{FZ24}. Let $\Omega \subset \mathbb{R}^d$, $d = 2,3$, be a bounded Lipschitz domain that either has a smooth boundary or is a convex polygon. For a given source term $f \in L^2(\Omega)$, we consider the problem: find $u : \Omega \subset \mathbb{R}^d \to \mathbb{C}$ such that 
\begin{align}\label{eq:nematic_helmholtz_korteweg}
      \Delta (\alpha \Delta u + \beta \nHn{u}) - \Delta u - k^2 u = f \text{ in } \Omega, 
\end{align}
where the parameters $\alpha > 0$ and $\beta \in \mathbb{R}$ are constitutive parameters and $\bm{n} \in \mathbb{R}^d$ is the constant nematic director. Further, we denote by $\hessian u$ the Hessian matrix of $u$ and by $\dn u \coloneqq \nabla u \cdot \nv$ the normal derivative, where $\nv$ is the outer unit normal to $\partial \Omega$. In contrast to the Helmholtz equation, where we typically model the acoustic pressure, the solution of \eqref{eq:nematic_helmholtz_korteweg} models the \emph{condensation} of the time-harmonic acoustic wave. 

We supplement \eqref{eq:nematic_helmholtz_korteweg} with mixed \emph{sound-soft} and \emph{impedance} boundary conditions. Let $\GammaD$ and $\GammaR$ be a non-intersecting partition of the boundary $\partial \Omega$ such that $\GammaD \cup \GammaR = \partial \Omega$, where we allow the cases where $\GammaR = \emptyset$ or $\GammaD = \emptyset$. For a specified impedance parameter $\theta  > 0$, we consider
\begin{subequations}\label{eq:BNDCondsNew}
   \begin{align}
      \begin{split}
      u &= 0 \quad \text{on } \Gamma_D, \\
      \dn \bigl(u - (\alpha \Delta u + \beta \nHn{u})\bigr) - i \theta u &= 0 \quad \text{on } \Gamma_R,
      \end{split}
      \label{eq:BNDCondsNew:a} \\
      \alpha \Delta u + \beta \nHn{u} &= 0 \quad \text{on } \partial \Omega.
      \label{eq:BNDCondsNew:b}
   \end{align}
\end{subequations}

The nematic Helmholtz--Korteweg equation \eqref{eq:nematic_helmholtz_korteweg} has first been derived in \cite{FZ24} by using the fact that nematic liquid crystals can be considered as nematic Korteweg fluids \cite{V09}. The impedance boundary conditions here are different from the original impedance boundary conditions $\dn (\alpha \Delta u + \beta \nHn{u}) = i \theta (\alpha \Delta u + \beta \nHn{u})$ considered in the derivation of the equation in \cite{FZ24}. Mathematically, this modification is motivated by the fact that the original boundary conditions are difficult to treat and require higher regularity assumptions, cf.~\cite[Rem.~2.2]{FvBZ25}. Instead, we impose \eqref{eq:BNDCondsNew:b}, which physically models the situation where the nematic liquid crystal is surrounded by an ideal gas. In this setting, we require \eqref{eq:BNDCondsNew:b} to ensure that the pressure and the density of the fluid match across the boundary $\partial \Omega$.

\begin{figure}[!htbp]
   \centering
   \includegraphics[width=0.98\textwidth]{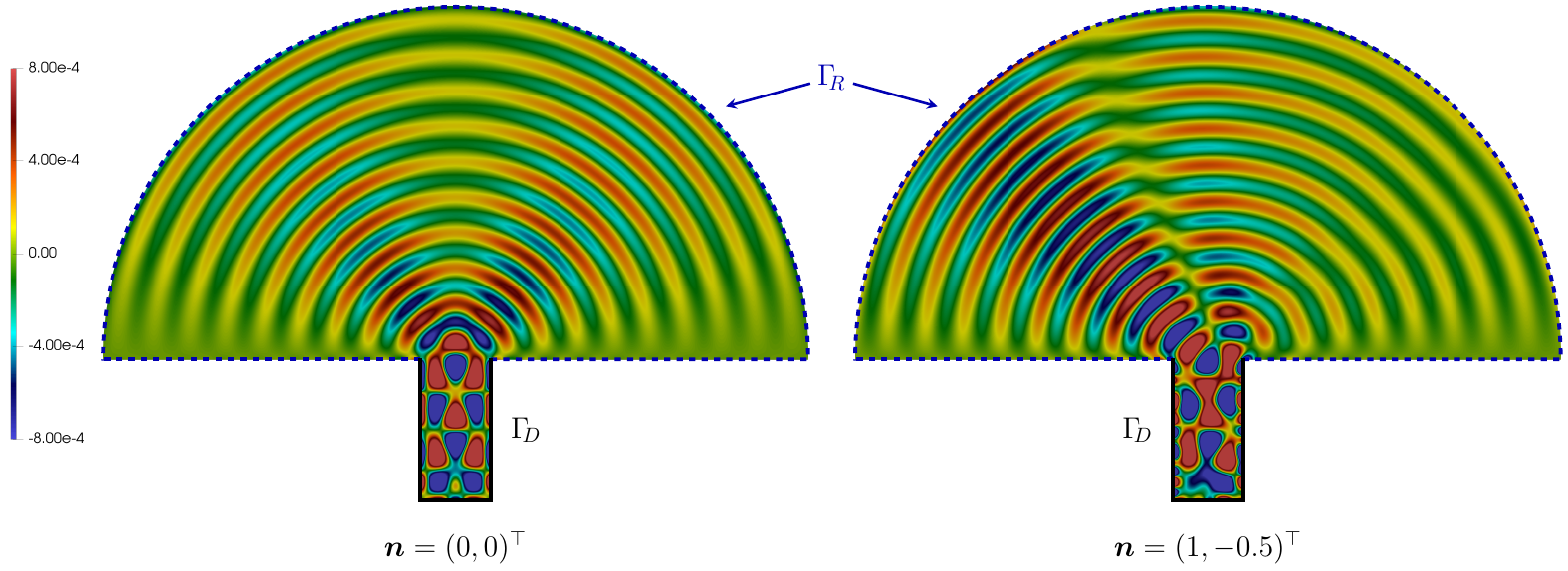}
   \caption{We compare the propagation of a Gaussian pulse into the depicted domain without (left) and with a nematic field (right). The solutions are computed with the method introduced in \Cref{subsec:DWF} for the parameters $\alpha = 5\cdot 10^{-3}$, $\beta = 10^{-3}$, $k = \theta = 75$ with polynomial degree $p = 7$. The anisotropy caused by the presence of the nematic field $\bm{n}$ is clearly visible in the right figure.}
   \label{fig:intro:waveguide}
\end{figure}

Beyond the difficulties inherent to the classical Helmholtz and biharmonic problems, the analysis of \eqref{eq:nematic_helmholtz_korteweg} is further complicated by the anisotropic fourth-order term. A first rigorous analysis of \eqref{eq:nematic_helmholtz_korteweg}, together with an $H^2$-conforming finite element discretization, was carried out in \cite{FvBZ25} for sound-soft and sound-hard (i.e., impedance boundary conditions with $\theta = 0$) boundary conditions. There, well-posedness is established by considering the problem as a perturbation of the isotropic case $\beta = 0$ for $\vert \beta \vert < C_{\beta = 0}$, where the threshold constant $C_{\beta = 0}$ depends on the stability constant of the isotropic problem. While a smallness condition on $\beta$ relative to $\alpha$ is physically reasonable \cite{V09}, the dependence on the isotropic stability constant is more delicate. This constant degenerates as $k$ approaches a resonant frequency of the isotropic problem, thus forcing $\beta \to 0$. This degeneracy is an artifact of the perturbation argument rather than a feature of the equation itself: the anisotropic problem possesses its own, distinct resonances (cf.~\Cref{rem:Eigenvalues}), and may remain well-posed at the isotropic resonant frequencies, exactly where the perturbation bound forces $\beta \to 0$.

The present work encompasses two major contributions. First, we prove that the problem \eqref{eq:nematic_helmholtz_korteweg} with boundary conditions \eqref{eq:BNDCondsNew} is Fredholm with index zero under an improved smallness condition on $\beta$, by exploiting the Cordes condition \cite{C56,SmearsSueli13,NSZ17Acta} to quantify the anisotropic effect of the nematic term.
Crucially, the resulting smallness condition on $\beta$ is independent of the stability constant of the isotropic case $\beta = 0$, and for example, is satisfied for all $\beta > - \alpha$ in the case of sound-soft boundary conditions on convex domains $\Omega \subset \mathbb{R}^2$. Assuming injectivity, the proof of which we leave for future work but justify numerically, this implies the well-posedness of \eqref{eq:nematic_helmholtz_korteweg} for a significantly broader range of $\beta \in \mathbb{R}$, independent of the stability constant of the isotropic case.

Second, we introduce a $C^0$-hybrid interior penalty ($C^0$-HIP) discretization of \eqref{eq:nematic_helmholtz_korteweg} that is considerably more flexible than comparable $C^1$-conforming discretizations. 
The implementation of $C^1$-conforming finite elements, such as the Argyris element \cite{Argyris} or the Hsieh--Clough--Tocher (HCT) macroelement \cite{HCT}, in curved or three-dimensional domains is notoriously difficult. Moreover, these elements require the polynomial degree to be sufficiently large; for example, the Argyris element requires polynomial degree $p \ge 5$ in two dimensions and $p \ge 9$ in three dimensions \cite{Z73}. To overcome these limitations, fourth-order problems are often discretized with non-conforming methods. For the biharmonic problem, fully discontinuous Galerkin methods have been considered in \cite{SM07,MSB07,MS03,D19} and hybrid high-order (HHO) methods in \cite{DE22,LTT25}. 

An alternative approach is the $C^{0}$-interior penalty ($C^0$-IP) method \cite{EGHLM02,BS05,Brenner2012,BN11}, which uses an $H^1$-conforming finite element space and discontinuous Galerkin techniques to treat the non-conformity with respect to $H^2$. We adopt this approach to discretize \eqref{eq:nematic_helmholtz_korteweg} because $C^0$-IP methods require fewer degrees of freedom than fully discontinuous Galerkin methods but still avoid the challenges of $C^1$-conforming methods. 

Furthermore, we use \emph{hybridization} \cite{CGL09,L10,DPE15,LS16} to improve the computational efficiency. In general, the idea of hybridization is to introduce additional facet unknowns that decouple the volume unknowns. The latter then can be eliminated through static condensation, greatly reducing the computational costs. While hybridization can also be applied to $C^0$-IP methods, the reduction is slightly less pronounced because static condensation only eliminates higher-order interior bubbles \cite[Rem.~3.3]{DE24C0HHO}. To the best of the author's knowledge, this has so far only been studied in \cite{DE24C0HHO} for the biharmonic problem and briefly discussed in \cite{F21C0HDG}. As such, the present work offers further insight into the hybridization of $C^0$-IP methods, with the problem \eqref{eq:nematic_helmholtz_korteweg} serving as a concrete example.

We prove that the proposed $C^0$-HIP discretization is stable for polynomial degree $p \ge 2$, whenever the continuous problem is well-posed. The key is a discrete G\r{a}rding inequality that reproduces the coercivity-up-to-compactness structure of the continuous problem at the discrete level. Casting the discretization into the framework of \cite{HLS22H1,Thesis_vB23,HLvB25}, this yields convergence of the approximations to the continuous solution under minimal regularity, that is, $u \in H^2(\Omega)$. We then derive optimal-order convergence rates under additional regularity assumptions on the continuous solution.

Finally, we provide numerical examples illustrating the capabilities of the proposed method. In particular, we present the first three-dimensional numerical simulations with \eqref{eq:nematic_helmholtz_korteweg} of the anisotropic effect of the nematic field.

\medskip
 
\textbf{Organization of the article.} In \Cref{sec:prelim}, we introduce the notation and the abstract framework that we use to analyze \eqref{eq:nematic_helmholtz_korteweg} and the proposed $C^{0}$-HIP discretization. Afterwards, we introduce and analyze the continuous weak formulation in \Cref{sec:weakForms}. The discretization is introduced in \Cref{subsec:DWF} and analyzed in \Cref{sec:NumAnalysis}. Finally, we present numerical examples in \Cref{sec:numex}.

\section{Notations and abstract framework}\label{sec:prelim}
This section collects the notation and the abstract theory used in the forthcoming sections.  Both our continuous and discrete analyses rest on a single mechanism: coercivity up to a compact perturbation, in the form of a G\r{a}rding inequality. At the continuous level this makes the operator Fredholm with index zero, see \Cref{lem:GardingsInequality}, so that well-posedness reduces to injectivity through the Fredholm alternative \cite[Thm.~8.2]{SBH19}. At the discrete level, reproducing the same structure uniformly yields asymptotic stability of the non-conforming discretization, see \Cref{thm:discreteGarding}.

\subsection{Notations} \label{sec:notation}
Let $\{ \Th \}_{h \in \mathcal{H}}$ be a family of shape-regular, simplicial triangulations of the bounded domain $\Omega \subset \mathbb{R}^d$, $d = 2,3$, indexed by the mesh size $h \coloneqq \max_{T \in \Th} h_T$, $h_T \coloneqq \operatorname{text}{diam}(T)$. We collect the mesh sizes in an index set $\mathcal{H} \subset \mathbb{R}_{> 0}$ and typically consider the case $h \to 0$. For $h \in \mathcal{H}$, we denote by $\Fh$ the collection of all facets of $\Th$, and by $\partial \Th$ the collection of all element boundaries. We assume that the triangulation respects the partition of the boundary $\partial \Omega = \GammaD \cup \GammaR$, that is, that each boundary facet $F \in \Fh \cap \partial \Omega$ can be uniquely associated with one of the sets $\GammaD$ or $\GammaR$. Then, we denote by 
\begin{align*}
   \Fh^{\mathcal{B}} \coloneqq \{ F \in \Fh : F \subset \Gamma_{\! \mathcal{B}} \},\ \mathcal{B} \in \{R,D\}, \text{ and } \Fh^{i} \coloneqq \{ F \in \Fh : F \cap \partial \Omega = \emptyset \},
\end{align*}
the collection of all facets associated with the respective parts of the boundary and the collection of all interior facets. The setup is visualized in \Cref{fig:notationAndSpace}. 

With abuse of notation, we denote by $\nv$ the unit outer normal to $\partial \Omega$ and the unit outer normal for each element $T \in \Th$. In the latter case, we assume the orientation of the normal vectors to be arbitrary but fixed in one direction and the mesh facets $F \in \Fh$ to be oriented accordingly. Across mesh interfaces $F \in \Fh$, we define the DG-jump of a scalar function $u$ as $\jump{u} \coloneqq u \vert_{T_1} - u \vert_{T_2}$ where $T_1, T_2 \in \Th$ are such that $T_1 \cap T_2 = F$.

We use standard notation for Sobolev spaces $H^s(\Omega)$ on the domain $\Omega$.
For any of the above defined collections $\Sigma_h \in \{ \Th, \Fh, \partial \Th \}$, we denote by $\mathbb{P}^p(\Sigma_h)$ the space of piecewise (discontinuous) polynomials of order up to $p$ on the elements of $\Sigma_h$, and for any broken Hilbert space $\mathbb{S}(\Sigma_h)$ defined piecewise on the collection $\Sigma_h$ (for example $L^2(\Th)$ or $H^s(\Th)$), we use the abbreviation 
\begin{align*}
   (\cdot,\cdot)_{\mathbb{S}(\Sigma_h)} \coloneqq \sum_{\sigma \in \Sigma_h} (\cdot,\cdot)_{\mathbb{S}(\sigma)}, \qquad \Vert \cdot \Vert_{\mathbb{S}(\Sigma_h)}^2 \coloneqq \sum_{\sigma \in \Sigma_h} \Vert \cdot \Vert_{\mathbb{S}(\sigma)}^2.
\end{align*}
Furthermore, we abbreviate $(\cdot,\cdot)_{\Omega} = (\cdot,\cdot)_{L^2(\Omega)}$ and $\Vert \cdot \Vert_{\Omega} = \Vert \cdot \Vert_{L^2(\Omega)}$ and similarly $(\cdot,\cdot)_{\Sigma_h} = (\cdot,\cdot)_{L^2(\Sigma_h)}$ for any collection $\Sigma_h \in \{ \Th, \Fh, \partial \Th \}$. 

The differential operators $\nabla u$, $\Delta u$, and $\hessian u$ are interpreted in a weak sense and for functions defined piecewise on the triangulation $\Th$, for example, elements of $\mathbb{P}^p(\Th)$, we consider these operators to be defined in a broken sense. We denote by $\dn u \coloneqq \nabla u \cdot \nv$ the normal derivative defined with respect to the normal vector $\nv$. 

Throughout the article, we denote by $C > 0$ a generic constant, which may change at each occurrence but is independent of the mesh size $h$ and the coefficients $\alpha$ and $\beta$.

\subsection{Abstract framework}
Let $V$ be a Hilbert space equipped with the inner product $(\cdot,\cdot)_{V}$ and let $\mathcal{L}(V)$ be the space of bounded linear operators acting on $V$. For a given bounded sesquilinear form $a(\cdot,\cdot) : V \times V \to \mathbb{C}$, we denote by $A \in \mathcal{L}(V)$ the associated bounded linear operator defined via $(Au,v)_{V} \coloneqq a(u,v)$, $u,v \in V$. Given $f \in L^2(\Omega)$, we study the well-posedness of the abstract problem
\begin{align}\label{eq:Abstract:CP}
   \text{Find } u \in V \text{ such that } a(u,v) = (f,v)_{\Omega} \text{ for all } v \in V. 
\end{align}
We employ the following G\r{a}rding inequality to show that the operator $A \in \mathcal{L}(V)$ is Fredholm with index zero. 

\begin{lemma}\label[lemma]{lem:GardingsInequality}
   Let $V, W$ be two Hilbert spaces such that $V \subset W$ with compact embedding $\iota : V \to W$. If there exist constants $C_V, C_W > 0$ such that 
   \begin{align}\label{eq:Abstract:Garding}
      \Re a(u,u) \ge C_V \Vert u \Vert^2_{V} - C_W \Vert u \Vert^2_{W} \qquad \forall u \in V, \tag{G}
   \end{align}
   then the associated operator $A \in \mathcal{L}(V)$ is a compact perturbation of a coercive operator and thus Fredholm with index zero.
\end{lemma}

\begin{proof}
   The G\r{a}rding inequality \eqref{eq:Abstract:Garding} implies that $A = B + K$ where $B \in \mathcal{L}(V)$ is coercive and $K \in \mathcal{L}(V)$ defined by $(Ku,v)_{V} \coloneqq - C_{W} (\iota u, \iota v)_{W}$, $u,v \in V$, is compact \cite[Thm.~5.20]{Spence14}. By the Lax--Milgram theorem, $B$ is a bounded isomorphism and thus $A$ is Fredholm with index zero.
\end{proof}

Consequently, the well-posedness of \eqref{eq:Abstract:CP} is equivalent to injectivity. For our intended application to \eqref{eq:nematic_helmholtz_korteweg}, we will apply \Cref{lem:GardingsInequality} with $V = H^2(\Omega) \hookrightarrow L^2(\Omega) = W$, leveraging the compactness of the Sobolev embedding $\iota : H^2(\Omega) \hookrightarrow L^2(\Omega)$.

We now consider approximations of the abstract problem \eqref{eq:Abstract:CP}, with the aim of transferring the structure induced by \eqref{eq:Abstract:Garding} to the discrete setting to obtain stability. Let $\seqh{V}$ be a sequence of finite-dimensional Hilbert spaces and $a_h(\cdot,\cdot) : V_h \times V_h \to \mathbb{C}$ be a bounded sesquilinear form. We denote by $\seqh{A}$, $A_h \in \mathcal{L}(V_h)$ the associated sequence of linear operators. We consider the discrete problem:
\begin{align}\label{eq:Abstract:DP}
   \text{Find } u_h \in V_h \text{ such that } a_h(u_h,v_h) = (f,v_h)_{\Omega} \text{ for all } v_h \in V_h.
\end{align}
The well-posedness of \eqref{eq:Abstract:DP} is equivalent to the well-known \emph{inf-sup condition}:
\begin{align}\label{eq:Abstract:inf-sup}
   \inf_{v_h \in V_h} \sup_{w_h \in V_h} \frac{\vert a_h(v_h,w_h) \vert}{\Vert v_h \Vert_{V_h} \Vert w_h \Vert_{V_h}} \ge \sconst > 0.
\end{align}
In the following, we call $a_h(\cdot,\cdot)$ or the sequence $\seqh{A}$ \emph{stable} if \eqref{eq:Abstract:inf-sup} is satisfied with constant $\sconst > 0$ independent of the mesh size $h$ and \emph{asymptotically stable} if there exists an $h_0 > 0$ such that stability holds for all $h \le h_0$.
The conforming\footnotemark \ Galerkin approximation of problems that are compact perturbations of coercive operators is well studied; see, for example, \cite{Dem94}, \cite[Thm.~4.2.9]{SS11}, \cite[Thm.~5.18]{Spence14}, \cite{BHP17}, or \cite[Sec.~8.9]{SBH19}. In particular, conforming Galerkin approximations of such problems are asymptotically stable if the continuous problem is well-posed.

\footnotetext{To be precise, we are considering a conforming Galerkin approximation, where the spaces $V_h$ are nested subspaces of $V$ such that $\bigcup_{h \in \mathcal{H}} V_h$ is dense in $V$. The discrete sesquilinear form is the restriction of $a(\cdot,\cdot)$ to the discrete subspace, i.e., $a_h \coloneqq a \vert_{V_h \times V_h}$.}

An abstract counterpart of this statement for non-conforming discretizations follows from the more general result in \cite[Thm.~3]{HLS22H1}. Therein, compact perturbations of T-coercive operators are considered (T-coercivity is equivalent to the inf-sup condition; see \Cref{appendix:Tcomp} for details). Intuitively, the statement is that we obtain asymptotic stability if we can transfer the structure to the discrete level consistently, that is, if we are uniformly (T-)coercive up to a compact sequence of perturbations.

The result has been applied in \cite{H23Hdiv,Thesis_vB23,HLvB25} to (hybrid) discontinuous Galerkin discretizations, and is formalized within the framework of \emph{discrete approximation schemes (DAS)} \cite{Vai81}, which allows for a non-conforming analysis without additional regularity assumptions on the continuous solution. To avoid unnecessary technicalities, we only introduce the parts of the framework that are directly relevant to our analysis and give additional details in \Cref{appendix:Tcomp}.
Since $V_h \not \subset V$, we cannot directly compare continuous and discrete elements. To circumvent this, we assume that there is a sequence of projection operators $\pi_h \in \mathcal{L}(V,V_h)$ such that 
\begin{align}\label{eq:pihCondition}
   \lim_{h \to 0} \Vert \pi_h u \Vert_{V_h} = \Vert u \Vert_{V} \text{ for all } u \in V. 
\end{align}
Then, we say that a sequence $\seqh{u}$, $u_h \in V_h$ \emph{converges} to an element $u \in V$ and write $u_h \overset{\pi_h}{\to} u$ if 
\begin{align}\label{eq:abstr:convergence}
   \lim_{h \to 0} \Vert u_h - \pi_h u \Vert_{V_h} = 0.
\end{align}
Similarly, we say that $\seqh{u}$, $u_h \in V_h$, \emph{converges weakly} to $u \in  V$ and write $u_h \overset{\pi_h}{\rightharpoonup} u$ if $\lim_{h \to 0} (u_h,\pi_h v)_{V_h} = (u,v)_{V}$ for all $v \in V$. As usual, any uniformly bounded sequence $\seqh{u}$, $u_h \in V_h$, has a weakly convergent subsequence \cite{Vai81}, and the strong and weak limits are unique.
Further, we require an alternative notion of consistency that holds \emph{asymptotically}, which is similar to the notion of consistency from \cite[Def.~5.9]{DPE12}.

\begin{definition}[Asymptotic consistency]\label[definition]{def:asympConsistency}
   We call a sesquilinear form $a_h (\cdot,\cdot) : V_h \times V_h \to \mathbb{C}$ \emph{asymptotically consistent} with $a(\cdot,\cdot) : V \times V \to \mathbb{C}$ if for all $u \in V$ and all uniformly bounded sequences $\seqh{v}$, $v_h \in V_h$, with  $v_h \overset{\pi_h}{\rightharpoonup} v$, $v \in V$, it holds that
   \begin{align}
      \lim_{h \to 0} a_h(\pi_h u,v_h) = a(u,v).
   \end{align}
\end{definition}

If a suitable projection $\pi_h$ exists and $a_h(\cdot,\cdot)$ is asymptotically consistent with $a(\cdot,\cdot)$, we say that the discretization $(a_h,\pi_h,V_h)$ constitutes a DAS of $(a,V)$. Then, we say that the sequence of sesquilinear forms $a_h(\cdot,\cdot)$ satisfies a discrete G\r{a}rding inequality if there exist constants $\widetilde{C}_{V}, \widetilde{C}_{W} > 0$ independent of $h$ such that 
\begin{align}\label{eq:Abstract:DiscrGarding}
   \Re a_h(u_h,u_h) \ge \widetilde{C}_{V} \Vert u_h \Vert^2_{V_h} - \widetilde{C}_{W} \Vert \iota_h u_h\Vert^2_{W}, \qquad \forall u_h \in V_h, \tag{G$_h$}
\end{align}
where $\iota_h : V_h \to W$ is a suitable operator. To show that \eqref{eq:Abstract:DiscrGarding} implies asymptotic stability by applying \cite[Thm.~3]{HLS22H1}, we need to verify that the perturbations in \eqref{eq:Abstract:DiscrGarding} form a compact sequence in the sense of DAS, cf.~\Cref{def:SeqCompactness}. This is guaranteed by the following assumption on $\iota_h$ that mirrors the compactness of the embedding $\iota: V \hookrightarrow W$.

\begin{assumption}\label[assumption]{ass:InclusionConsistentlyCompact}
   There exist a uniformly bounded sequence of mappings $\iota_h : V_h \to W$ and an operator $J \in \mathcal{L}(V,W)$ such that for all uniformly bounded sequences $\seqh{v}$ with $v_h \overset{\pi_h}{\rightharpoonup} v$, $v\in V$, we have that $\iota_h v_h \to Jv$ strongly in $W$. 
\end{assumption}

\Cref{ass:InclusionConsistentlyCompact} does \emph{not} require that $V_h \subset W$, which is important for our application in \Cref{sec:NumAnalysis} because the hybrid ansatz space introduced in \Cref{subsec:HIP} consists of tuples of volume and facet unknowns. Thus, it has a different structure from $W$ and is formally not a subspace. Finally, the following result shows that we obtain asymptotic stability, provided that we have a DAS and can mirror the structure of '$\text{\eqref{eq:Abstract:Garding}}+\text{compactness}$' suitably on the discrete level.

\begin{theorem}\label{thm:discreteGarding}
   Suppose that $a(\cdot,\cdot)$ admits the structure from \Cref{lem:GardingsInequality} and that the problem \eqref{eq:Abstract:CP} is well-posed. Let $(a_h,\pi_h,V_h)$ be a DAS of $(a,V)$ and $V_h$ satisfy \Cref{ass:InclusionConsistentlyCompact}. If $a_h(\cdot,\cdot)$ satisfies \eqref{eq:Abstract:DiscrGarding},
   then $a_h(\cdot,\cdot)$ is asymptotically stable on $V_h \times V_h$.
\end{theorem}

\begin{proof}
   We provide a proof in \Cref{appendix:Tcomp}. \Cref{ass:InclusionConsistentlyCompact} implies that the sequence of perturbations in \eqref{eq:Abstract:DiscrGarding} is compact in a suitable sense and then the claim follows from \cite[Thm.~3]{HLS22H1} and \cite[Lem.~1]{HLS22H1}.
\end{proof}

\begin{remark}[Asymptotic threshold]
   For conforming finite element discretizations, the G\r{a}rding inequality is classically used in combination with the Schatz argument \cite{S74} to obtain an explicit characterization of the asymptotic threshold $h_0$ in terms of the \emph{adjoint approximation property}. This argument does \emph{not} immediately carry over to our setting, because the discretization is non-conforming and our notion of consistency itself is asymptotic. Thus, the threshold $h_0$ in \Cref{thm:discreteGarding} is not explicit.
\end{remark}

\section{The continuous problem}\label{sec:weakForms}
In \Cref{subsec:WF}, we introduce the weak formulation of \eqref{eq:nematic_helmholtz_korteweg}. In \Cref{subsec:WP}, we show that the associated operator is Fredholm with index zero provided that $\beta$ is sufficiently small, and we discuss the injectivity of the problem. 

\subsection{Weak formulation}\label{subsec:WF}
In preparation for the forthcoming analysis, we rewrite the highest-order term in \eqref{eq:nematic_helmholtz_korteweg} as
\begin{align}\label{eq:highestOrderTerm}
   \Delta (\alpha \Delta u + \beta \nHn{u})  = \Delta(\FOtermu), \qquad  \mathbb{A} \coloneqq \alpha \Idm + \beta \bm{n} \otimes \bm{n} \in \mathbb{R}^{d \times d},
\end{align}
where $\frobip{\mathbb{A}}{\mathbb{B}} = \operatorname{tr}(\mathbb{A}^\top \mathbb{B})$, $\Am, \mathbb{B} \in \mathbb{R}^{d \times d}$, denotes the Frobenius inner product. Throughout, we set $\gamma \coloneqq \operatorname{tr} \Am / \vert \Am \vert^2_{F} > 0$, where $\vert \cdot \vert_{F}$ is the Frobenius norm; the role of this scaling factor becomes apparent in \Cref{subsec:WP}. Then, we define the Hilbert space and inner product
\begin{align}\label{eq:def:V}
   V \coloneqq \{ v \in H^2(\Omega) : v = 0 \text{ on } \GammaD \}, \qquad 
   (u,v)_V \coloneqq \gamma^{-1} (\Delta u,\Delta v)_{\Omega} + (u,v)_{H^1(\Omega)},
\end{align}
with induced norm $\Vert u \Vert_V^2 \coloneqq (u,u)_V$. To derive the weak formulation of \eqref{eq:nematic_helmholtz_korteweg}, we multiply by a test function $v \in V$, integrate by parts, and use the boundary conditions \eqref{eq:BNDCondsNew}, which yields
\begin{align}\label{eq:PartialIntegration}
   (\Delta(\FOtermu),v)_{\Omega} - (\Delta u, v)_{\Omega} = (\FOtermu,\Delta v)_{\Omega} + (\nabla u, \nabla v)_{\Omega} - i \theta (u,v)_{\GammaR}.
\end{align}
Accordingly, we define the sesquilinear form $a(\cdot,\cdot) : V \times V \to \mathbb{C}$ as 
\begin{align}\label{eq:cont:a}
   a(u,v) &\coloneqq (\FOtermu,\Delta v)_{\Omega} + (\nabla u, \nabla v)_{\Omega} - k^2(u,v)_{\Omega} - i \theta (u,v)_{\GammaR}.
\end{align}   
The weak formulation of \eqref{eq:nematic_helmholtz_korteweg} then reads as: find $u \in V$ such that 
\begin{equation}\label{eq:Cont:WF}
   a(u,v) = (f,v)_{\Omega} \quad \forall v \in V. \tag{WF}
\end{equation}

\subsection{Fredholmness of the continuous problem}\label{subsec:WP}
The highest-order terms in \eqref{eq:cont:a} vanish when testing with a harmonic function, but the nematic term couples with the full Hessian. Thus, we need to ensure that all second order derivatives can be controlled by the Laplacian. In other words, we need the following elliptic regularity result as a critical ingredient of the forthcoming analysis. To this end, we assume the domain $\Omega \subset \mathbb{R}^d$ to either have a smooth boundary or to be a convex polygon. For two-dimensional domains with sound-soft boundary conditions, it is sufficient if $\Omega$ is polygonal. 

\begin{lemma}[Elliptic regularity]\label[lemma]{lem:ellipticReg}
   Suppose that $\Omega$ is sufficiently regular in the above sense. Then, there exists a constant $\CR \ge 1$ such that 
   \begin{align}\label{eq:ellipticReg}
      \Vert \hessian u \Vert_{\Omega} \le \CR \left( \Vert \Delta u \Vert_{\Omega} + \Vert u \Vert_{\Omega} \right), \qquad \forall u \in V. \tag{ER}
   \end{align}
   If $\partial \Omega = \GammaD$ and $\Omega \subset \mathbb{R}^d$, $d = 2,3$, is convex or $\Omega \subset \mathbb{R}^2$ is polygonal, the Miranda--Talenti inequality holds, i.e.,
   \begin{align}\label{eq:MT}
      \Vert \hessian u \Vert_{\Omega} \le \Vert \Delta u \Vert_{\Omega}, \qquad \forall u \in V = H^2(\Omega) \cap H^1_0(\Omega). \tag{MT}
   \end{align}
\end{lemma}

\begin{proof}
   Classical elliptic regularity results \cite[Thm.~2.3.3.2]{Grisvard} yield that $\Vert u \Vert_{H^2} \le C( \Vert \Delta u \Vert_{\Omega} + \Vert u \Vert_{H^1(\Omega)})$. Ehrling's lemma \cite[Thm.~7.30]{RR04}, together with the compactness of the embeddings $H^2 \hookrightarrow H^1 \hookrightarrow L^2$, allows us to replace the $H^1$-norm with the $L^2$-norm. The special case \eqref{eq:MT} is the \emph{Miranda--Talenti} inequality, see, for instance, \cite{MPS00} or \cite[Rem.~2.37]{NSZ17Acta}, where it is shown that the convexity assumption is not necessary for two-dimensional polygonal domains.
\end{proof}

An immediate consequence of \eqref{eq:ellipticReg} is that the norm $\Vert \cdot \Vert_{V}$ is equivalent to the $H^2$-norm and that the sesquilinear form $a(\cdot,\cdot)$ is bounded with respect to $\Vert \cdot \Vert_{V}$.

The principal term $(\FOtermu,\Delta v)_{\Omega}$ has the same structure that frequently appears in the analysis of second-order PDEs in non-divergence form, see, for instance, \cite{SmearsSueli13,NSZ17Acta,S26}.
A crucial ingredient in the analysis of those problems is the \emph{Cordes condition} \cite{C56}, which limits the anisotropy induced by the matrix $\Am$. 
We work directly with its quantitative form and require that $\vert \gamma \Am - \Idm \vert_{F} < 1$. The relevance of this quantity stems from the pointwise estimate
\begin{align}\label{eq:cordesEstimate}
   \vert \gamma \FOtermu - \Delta u \vert \le \vert \gamma \Am - \Idm \vert_{F} \vert \hessian u \vert_{F}, \quad \forall u \in V, \text{ a.e. in } \Omega. 
\end{align}
In the setting of \cite{SmearsSueli13,NSZ17Acta,S26}, the right-hand side of \eqref{eq:cordesEstimate} can be estimated with \eqref{eq:MT} and $\vert \gamma \Am - \Idm \vert_{F} < 1$ is sufficient to obtain \eqref{eq:Abstract:Garding}. In the general case, the elliptic regularity constant enters the condition.

\begin{lemma}[G\r{a}rding inequality]\label[lemma]{lem:weakcoercivity}
   Assume that $\CR \vert \gamma \Am - \Idm \vert_{F} < 1$. Then, there exist constants $C_{V}, C_{L^2}  > 0$ such that 
   \begin{align*}
      \Re a(u,u) \ge C_V \Vert u \Vert^2_{V} - C_{L^2} \Vert u \Vert^2_{\Omega} \qquad \forall u \in V.
   \end{align*} 
\end{lemma}

\begin{proof} 
   By assumption, $\dCr \coloneqq \CR \vert \gamma \Am - \Idm \vert_{F} < 1$, so that we can fix $\epsilon > 0$ with $\dCr(1+\epsilon) < 1$. The Cauchy--Schwarz inequality, the pointwise estimate \eqref{eq:cordesEstimate}, \Cref{lem:ellipticReg}, and a weighted Young's inequality yield that
   \begin{equation}\label{eq:ContCoercivity:CordesConsequences}
      \begin{aligned}
         \left\vert (\gamma \FOtermu - \Delta u, \Delta u)_{\Omega} \right\vert &\le \vert \gamma \Am - \Idm \vert_{F} \Vert \hessian u \Vert_{\Omega} \Vert \Delta u \Vert_{\Omega} \\
         &\le \CR \vert \gamma \Am - \Idm \vert_{F} \left( \Vert \Delta u \Vert_{\Omega}^2 + \Vert u \Vert_{\Omega} \Vert \Delta u \Vert_{\Omega} \right) \\
         &\le \dCr (1 + \epsilon) \Vert \Delta u \Vert^2_{\Omega} + \frac{\dCr}{4 \epsilon} \Vert u \Vert_{\Omega}^2.
      \end{aligned}
   \end{equation}
   Consequently, we obtain that 
   \begin{equation}\label{eq:ContCoercivity:AD2Estimate}
      \begin{aligned}
         \Re (\FOtermu,\Delta u)_{\Omega} &= \frac{1}{\gamma} \Vert \Delta u \Vert_{\Omega}^2 + \frac{1}{\gamma} \Re (\gamma \FOtermu - \Delta u, \Delta u)_{\Omega} \\
         &\stackrel{\eqref{eq:ContCoercivity:CordesConsequences}}{\ge} \frac{1}{\gamma} \left( 1 - \dCr (1 + \epsilon) \right) \Vert \Delta u \Vert_{\Omega}^2 - \frac{\dCr}{4 \gamma \epsilon} \Vert u \Vert_{\Omega}^2.
      \end{aligned}
   \end{equation}
   Since the Robin term in \eqref{eq:cont:a} is purely imaginary, adding and subtracting $\Vert u \Vert_{\Omega}^2$ yields that 
   \begin{align*}
      \Re \{ a(u,u) \} &\ge \frac{ 1 - \dCr (1 + \epsilon)}{\gamma} \Vert \Delta u \Vert^2_{\Omega} + \Vert \nabla u \Vert^2_{\Omega} + \Vert u \Vert^2_{\Omega} - (1 + k^2 + \dCr (4 \gamma \epsilon)^{-1}) \Vert u \Vert^2_{\Omega} \\
      &\ge \min \{ 1 - \dCr (1 + \epsilon), 1 \} \Vert u \Vert^2_{V} - C_{L^2} \Vert u \Vert^2_{\Omega},
   \end{align*}
   which proves the claim with $C_V = 1 - \dCr(1+\epsilon)$ and $C_{L^2} = 1 + k^2 + \dCr (4 \gamma \epsilon)^{-1}$.  
\end{proof}

Since $V \hookrightarrow W \coloneqq L^2(\Omega)$ with compact embedding due to Rellich's theorem, \Cref{lem:GardingsInequality} applies under the assumptions of \Cref{lem:weakcoercivity}, so that the operator associated with \eqref{eq:cont:a} is Fredholm with index zero. Before discussing injectivity, we investigate the smallness assumption on $\Am$. 

For $\alpha > 0$ and $\beta > - \alpha$, the matrix $\Am \coloneqq \alpha \Idm + \beta \bm{n} \otimes \bm{n} \in \mathbb{R}^{d \times d}$ is symmetric positive definite with 
\begin{align}\label{eq:AmProps}
   (\lambda_1, \dots, \lambda_{d-1}, \lambda_d) = (\alpha, \dots, \alpha,\alpha + \beta),  \qquad \vert \Am \vert^2_{F} = d \alpha^2 + 2 \alpha \beta + \beta^2, \qquad \operatorname{tr} \Am = d \alpha + \beta.
\end{align}
This allows us to derive an explicit smallness condition on $\beta$ in terms of $\alpha$, the elliptic regularity constant $\CR$, and the dimension $d$.

\begin{lemma}\label[lemma]{lem:alphabetaConditionImpliesCordes}
   We have that $\vert \gamma \Am -\Idm \vert^2_{F} = (d-1) \beta^2 \vert \Am \vert_{F}^{-2}$.
   Further, $\CR \vert \gamma \Am - \Idm \vert_{F} < 1$ if and only if 
   \begin{align}\label{eq:alphabetaCondition}
      \beta \in (\alpha \mu^{-1}(1-\sqrt{1+d\mu}),\alpha \mu^{-1}(1+\sqrt{1+d\mu})), \qquad \mu \coloneqq \CR^2(d-1)-1,
   \end{align}
   where, for $\mu = 0$, the interval is understood in the sense of the limit $\mu \to 0^+$.
\end{lemma}
\begin{proof}
   We assume $\mu > 0$ and calculate with the definition of $\gamma$ that
   \begin{align*}
      \vert \Am \vert^2_{F} \vert \gamma \Am - \Idm \vert_{F}^2 = \vert \Am \vert^2_{F} \left( \gamma^2 \vert \Am \vert^2_{F} - 2 \gamma \operatorname{tr}(\Am) + d \right) = d \vert \Am \vert^2_{F} - (\operatorname{tr}(\Am))^2. 
   \end{align*}
   After inserting \eqref{eq:AmProps}, some algebraic manipulations yield the first claim. Using this expression, we have that $\CR^2 \vert \gamma \Am - \Idm \vert^2_{F} < 1$ if and only if 
   \begin{align*}
      \CR^2(d-1)\beta^2 < \vert \Am \vert^2_{F} = d \alpha^2 + 2 \alpha \beta + \beta^2 \quad \Leftrightarrow \quad q(\beta) \coloneqq \mu \beta^2 - 2\alpha \beta - d \alpha^2 < 0.
   \end{align*}
   Since $\mu > 0$ for $d \ge 2$, $\CR \ge 1$, the inequality is satisfied for $\beta_{-} < \beta < \beta_{+}$, where $\beta_{-}$ and $\beta_{+}$ are the negative and positive roots of the quadratic polynomial $q(\cdot)$. Solving for the roots with the quadratic formula yields \eqref{eq:alphabetaCondition}. The special case $\mu = 0$ only occurs when $d = 2$ and $\CR = 1$, in which case we interpret $\mu = 0$ as a limit $\mu \to 0$.
\end{proof}

For large elliptic regularity constants $\CR$, the condition \eqref{eq:alphabetaCondition} enforces that $\vert \beta \vert = \mathcal{O}(\alpha \CR^{-1})$. For convex domains with $\GammaD = \partial \Omega$, we may take $\CR  = 1$ due to \eqref{eq:MT}.  In this case, for $d = 2$ we have $\mu = 0$ and hence $\beta \in (-\alpha, \infty)$; that is, whenever \eqref{eq:MT} applies in the two-dimensional setting, the smallness requirement is satisfied for every admissible $\beta > -\alpha$. For $d = 3$, we obtain $\beta \in (-\alpha, 3\alpha)$. In the general setting, we conclude from \Cref{lem:alphabetaConditionImpliesCordes} and \Cref{lem:weakcoercivity} that the operator associated with \eqref{eq:cont:a} is Fredholm with index zero provided that $\beta$ satisfies \eqref{eq:alphabetaCondition}. To establish the well-posedness of the problem, we need to show that the operator is injective. We recall the well-posedness result from \cite{FvBZ25}.

\begin{lemma}[Injectivity]
   The operator associated with \eqref{eq:cont:a} is injective for all non-resonant wave numbers if $\vert \beta \vert < C_{\beta = 0}$, where $C_{\beta = 0}$ depends on the stability constant of the isotropic problem. 
\end{lemma}

\begin{proof}
   The statement follows from \cite[Thm.~3.2]{FvBZ25}. Essentially, the argument is based on a classical perturbation argument \cite{Kato66}, where we first show that the isotropic problem with $\beta = 0$ is well-posed and then infer well-posedness of \eqref{eq:Cont:WF} for all $\beta$ such that $\vert \beta \vert < C_{\beta = 0}$.
\end{proof}

The constant $C_{\beta = 0}$ degenerates as $k^2$ approaches an eigenvalue of the isotropic operator $\alpha \Delta^2 - \Delta$, forcing $\beta \to 0$. However, we expect the anisotropic problem with $\beta > 0$ to have its own set of resonant frequencies $\mathcal{K}_{\text{res}}$ related to the eigenvalues of the operator $\Delta(\frobip{\Am}{\hessian}) - \Delta$, cf.~also \Cref{rem:Eigenvalues}. We therefore make the following assumption. 

\begin{assumption}[Injectivity under \eqref{eq:alphabetaCondition}]\label[assumption]{ass:injectivity}
   The operator associated with \eqref{eq:cont:a} is injective for all $\beta$ satisfying \eqref{eq:alphabetaCondition} and all non-resonant wave numbers $k^2 \not \in \mathcal{K}_{\text{res}}$.
\end{assumption}

A proof of the injectivity assumed in \Cref{ass:injectivity} is beyond the scope of this article. The non-symmetric volume term introduces compounding difficulties: a spectral analysis of non-symmetric perturbations of self-adjoint operators requires additional technical machinery \cite{CT16}, and classical uniqueness arguments for impedance problems, such as unique continuation \cite[Lem.~35.10]{EG_FE2} or Morawetz--Rellich identities \cite{GPS19}, do not carry over directly  because the nematic term contributes to the imaginary part $\Im a(u,u)$.

\begin{remark}[Error blow-up at resonant $k^2$]\label{rem:Eigenvalues}
   Let $\Omega = [0,1]^2$, $\GammaD = \partial \Omega$, and $\bm{n} = (1,0)^{\top}$. A direct calculation shows that the functions $e_{mn} \coloneqq \sin(\pi m x) \sin(\pi ny)$, $m, n \in \mathbb{N}_{\ge 1}$ are eigenfunctions of $\Delta(\frobip{\Am}{\hessian}) - \Delta$ with associated eigenvalues
   \begin{align}\label{eq:EValsUnitSquare}
      \lambda_{mn}^{\beta} \coloneqq \pi^4 \left( \alpha (m^2 + n^2)^2 + \beta m^2(m^2 + n^2) \right) + \pi^2 (m^2 + n^2).
   \end{align}
   \Cref{tab:eigenvalues} lists the first four eigenvalues for several values of $\beta$. For larger values of $\beta$, the spectrum shifts significantly; in particular, $\lambda_{mn}^{\beta} \neq \lambda_{nm}^{\beta}$ for $\beta > 0$. Moreover, in \Cref{fig:EVBlowUp} we observe that the discretization error blows up close to resonant wave numbers $k^2 \approx \lambda_{mn}^{\beta}$. For larger values of $\beta$, the blow-up occurs when $k^2 \approx \lambda_{mn}^{\beta}$, but not when $k^2 \approx \lambda_{mn}^{0}$. This supports \Cref{ass:injectivity}, as the stability of the problem with $\beta > 0$ does not appear to deteriorate when $C_{\beta = 0} = \mathcal{O}(\vert k^2 - \lambda^0_{mn} \vert)$ becomes small.
   \begin{figure}[!htbp]
      \centering
      \begin{subfigure}{0.47\textwidth}
          \centering
          \includegraphics[width=\linewidth]{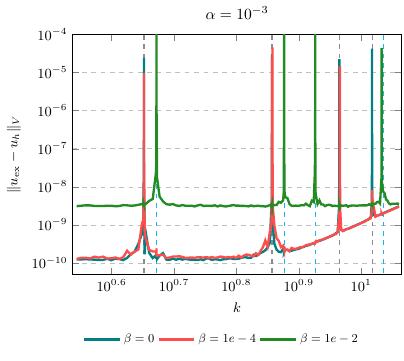}
          \caption{Discretization error for $k \in (3.5,11.5)$ and $\beta \in \{0, 10^{-4}, 10^{-2}\}$.}
          \label{fig:EVBlowUp}
      \end{subfigure}
      \hfill
      \begin{subfigure}{0.47\textwidth}
          \centering
          \vspace*{-2cm}
          \begin{tabularx}{\linewidth}{l *{3}{>{\raggedleft\arraybackslash}X}}
              \toprule
              & $\beta = 0$ & $\beta = 10^{-4}$ & $\beta = 10^{-2}$ \\
              \midrule
              $\lambda_{11}^{\beta}$ & $20.128$  & $20.148$ & $22.077$ \\
              $\lambda_{12}^{\beta}$ & $51.783$ & $51.831$ & $56.653$ \\
              $\lambda_{21}^{\beta}$ & $51.783$ & $51.978$ & $71.265$ \\
              $\lambda_{22}^{\beta}$ & $85.191$ & $85.502$ & $116.361$ \\
              \bottomrule
          \end{tabularx}
          \vspace{2cm}
          \caption{First four eigenvalues computed with \eqref{eq:EValsUnitSquare} for $\alpha = 10^{-3}$ and different values of $\beta$.}
          \label{tab:eigenvalues}
      \end{subfigure}
      \caption{Error blow-up at $k^2 \approx \lambda_{mn}^{\beta}$ and corresponding eigenvalues for different values of $\beta$ with fixed $\alpha = 10^{-3}$. We discretize \eqref{eq:Cont:WF} with the discretization presented in \Cref{subsec:HIP} on a fixed mesh with polynomial degree $p = 7$ and calculate the error against the reference solution \eqref{eq:planewave}.}
      \label{fig:EVals:combined}
  \end{figure}
\end{remark}

\begin{remark}[Smoothly varying coefficients]
  With minor modifications, the analysis, and in particular \Cref{lem:weakcoercivity}, extends to smoothly varying bounded coefficients $\alpha, \beta \in C^{\infty}(\overline{\Omega})$ with $0 < C_{\min}^{\alpha} \le \alpha(x) \le C_{\max}^{\alpha}$ and $C_{\min}^{\beta} \le \beta(x) \le C_{\max}^{\beta}$. In this case, $\gamma$ becomes a function of $x$, and we require that $\alpha + \beta > 0$ and $\vert \gamma \Am - \Idm \vert_{F} < 1$ hold pointwise a.e. in $\Omega$.
\end{remark}

\section{The discrete setting}\label{subsec:DWF}
In \Cref{subsec:HIP}, we introduce a discretization of \eqref{eq:nematic_helmholtz_korteweg} with a non-conforming $C^{0}$-hybrid interior penalty ($C^{0}$-HIP) method. In \Cref{subsec:Projections}, we construct the projection operators that allow us to apply the framework from \Cref{sec:prelim} and to show convergence estimates in \Cref{thm:convergenceEst}.

\subsection{$C^{0}$-hybrid interior penalty discretization}\label{subsec:HIP}
For polynomial degree $p \ge 2$, we define the discrete space $V_h \coloneqq V_{\Th} \times V_{\Fh}$ with
\begin{align}\label{eq:def:VhSpaces}
   V_{\Th} \coloneqq \mathbb{P}^p(\Th) \cap \{ v \in H^1(\Omega) : v = 0 \text{ on } \GammaD \}, \qquad
   V_{\Fh} \coloneqq \mathbb{P}^{p-1}(\Fh),
\end{align}
that is, $V_{\Th}$ is a Lagrange finite element space of degree $p$ on the domain $\Omega$ and $V_{\Fh}$ is a space of piecewise polynomials of degree $p-1$ on the collection of facets $\Fh$. The facet unknowns represent the normal derivative of the volume unknowns.
\Cref{fig:notationAndSpace} shows the degrees of freedom of $V_h$ for $p = 4$ and recalls the notation introduced in \Cref{sec:notation}.

\colorlet{Ldofs}{teal!90!black}
\colorlet{Fdofs}{red!75!black}

\begin{figure}[!htbp]
   \begin{center}
      \resizebox{0.99\textwidth}{!}{
      \begin{tikzpicture}
         \node at (0,0) {\includegraphics[width=0.98\textwidth]{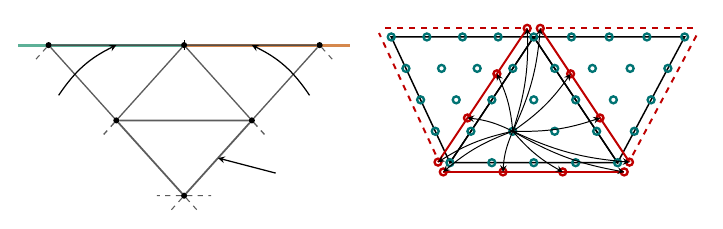}};
         \node at (-0.45,1.65) {\footnotesize $\GammaD$};
         \node at (-5.8,1.65) {\footnotesize$\GammaR$};
         \node at (-6.2,0.15) {\footnotesize $\in \Fh^{R}$};
         \node at (-0.75,0.15) {\footnotesize $\in \Fh^{D}$};
         \node[draw, circle,very thick,inner sep=1.25pt,color=Ldofs] (ID) at (1.65,-1.8) {};
         \node at ($(ID)+(1.45,-0.0125)$) {\footnotesize interior dof ($p = 4$)};
         \node[draw, circle,very thick,inner sep=1.25pt,color=Fdofs] at ($(ID)+(0,-0.5)$) {};
         \node at ($(ID)+(0,-0.5)+(1.29,-0.0125)$) {\footnotesize facet dof ($p = 3$)};
         \node at (-3.6,-0.4) {\footnotesize $T \in \Th$};
         \node at (-1,-1.3) {\footnotesize $\partial T \in \partial \Th$};
      \end{tikzpicture}
      }
   \end{center}
   \caption{In the left part of the figure, we recall the setup and notation as introduced in \Cref{sec:notation}. In the right part of the figure, we display the degrees of freedom (dofs) of $V_h$ for $p = 4$. In two dimensions, the volume unknowns have $(p+2)(p+1)/2$ dofs per element, where neighboring elements share the dofs on the common facet, and the facet unknowns have $p$ dofs per facet. The interior element bubbles couple only through the facet unknowns and can be eliminated by static condensation.}
   \label{fig:notationAndSpace}
\end{figure}
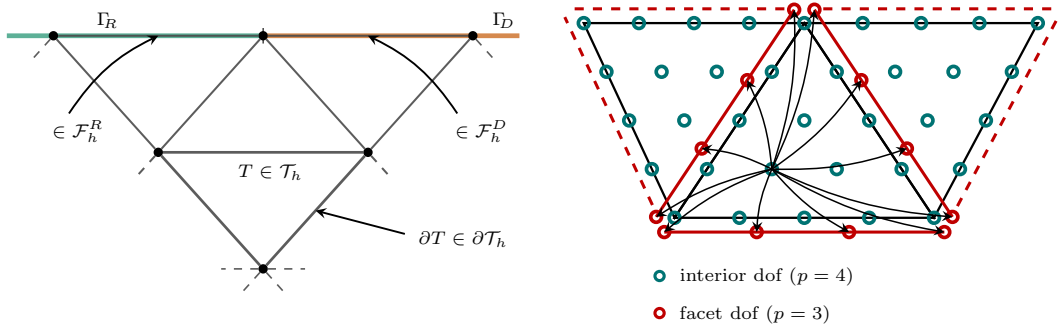

We write $\uh = (\uvol, \ufac) \in V_h$ and $\vh = (\vvol, \vfac) \in V_h$, where $\uvol,\vvol \in V_{\Th}$ denote the volume components and $\ufac, \vfac \in V_{\Fh}$ denote the facet components. Then, we define the HDG-jump operator by
\begin{align}
   \hdgjump{\dn u_h} \coloneqq \dn \uvol - \ufac, \qquad (\uvol,\ufac) \in V_h,
\end{align} 
and the stabilization term and associated semi-norm by
\begin{align}
   \sterm_h(\uh,\vh) &\coloneqq ( h^{-1} \hdgjump{\dn u_h}, \hdgjump{\dn v_h})_{\partial \Th}, \quad \vert \uh \vert_{\sterm_h}^2 \coloneqq \sterm_h(\uh,\uh), \qquad \uh, \vh \in V_h.
\end{align}
For a stabilization parameter $\stab > 0$, we then define the discrete sesquilinear form $a_h(\cdot,\cdot) : V_h \times V_h \to \mathbb{C}$ as 
\begin{equation}
   \begin{aligned}
      a_h(\uh,\vh) &\coloneqq (\FOtermuh,\Delta \vvol)_{\Th} - (\FOtermuh,\hdgjump{\dn \vvol})_{\partial \Th} + \stab \sterm_h(\uh,\vh) \\
      &\qquad+ (\nabla \uvol, \nabla \vvol)_{\Omega}  - k^2(\uvol,\vvol)_{\Omega} - i \theta (\uvol,\vvol)_{\Fh^{R}}. 
   \end{aligned}
\end{equation} 
In contrast to the classical symmetric interior penalty method, we do not add a symmetrization term, since the volume term itself is non-symmetric; this is similar to the incomplete interior penalty method, cf.~\cite{DSW04,DPE12}. The discrete weak formulation reads as: find $\uh \in V_h$ such that 
\begin{align}\label{eq:discrete_weak_form}
   a_h(\uh,\vh) = (f,\vvol)_{\Omega} \quad \forall \vh \in V_h. \tag{WF$_h$}
\end{align}

Similar to the reconstruction operator used in HHO methods \cite{DPE15,DE24C0HHO}, we define a lifting operator $\lift : V_h \to \mathbb{P}^{p-2}(\Th)$ by
\begin{align}\label{eq:def:lifting}
   (\psi_h,\lift \vh)_{\Th} \coloneqq - (\psi_h,\hdgjump{\dn \vvol})_{\partial \Th} \qquad \forall \psi_h \in \mathbb{P}^{p-2}(\Th).
\end{align}
By thee Cauchy--Schwarz and the discrete trace inequality, we have that 
\begin{align}\label{eq:LiftBounded}
   \Vert \lift \vh \Vert_{\Th} \le \Clift  \vert \vh \vert_{\sterm_h} \qquad \forall \vh \in V_h.
\end{align}
We further define a discrete Laplacian element-wise for $T \in \Th$ through $(\Delta_h \vh) \vert_{T} \coloneqq \Delta \vvol + \lift \vh 
\in \mathbb{P}^{p-2}(\Th)$, $\vh \in V_h$. Since $\FOtermuh \in \mathbb{P}^{p-2}(\Th)$, we may choose $\psi_h = \FOtermuh$ in \eqref{eq:def:lifting} and rewrite the sesquilinear form $a_{h}(\cdot,\cdot)$ as 
\begin{equation}\label{eq:ahLift}
   \begin{aligned}
      a_{h}(\uh,\vh) &= (\FOtermuh,\Delta_h \vh)_{\Th} + \stab \sterm_h(\uh,\vh) \\
      &\qquad+ (\nabla \uvol, \nabla \vvol)_{\Omega} - k^2(\uvol,\vvol)_{\Omega} - i \theta (\uvol,\vvol)_{\Fh^{R}}.
   \end{aligned}
\end{equation}

Finally, we equip the discrete space $V_h$ with the following inner product
\begin{align*}
   (\uh,\vh)_{V_h} \coloneqq \gamma^{-1} (\Delta_h \uh,\Delta_h \vh)_{\Th} + (\uvol,\vvol)_{H^1(\Omega)} + \gamma^{-1} \sterm_h(\uh,\vh), 
\end{align*}
and define the induced discrete norm through $\Vert \uh \Vert^2_{V_h} \coloneqq (\uh,\uh)_{V_h}$, $\uh \in V_h$.

\subsection{Projection operators}\label{subsec:Projections}
To apply the theory from \Cref{sec:prelim}, in particular \Cref{thm:discreteGarding}, we need to construct a suitable projection operator that satisfies \eqref{eq:pihCondition}.
Since $u \in V \subset H^2(\Omega)$, the jump of the normal derivative across interfaces is well-defined and continuous. Formally, we identify $u \in V$ with a tuple $\u \in V \times L^2(\Fh)$ via the mapping 
\begin{align*}
   V \to V \times L^2(\Fh), \quad u \mapsto \u \coloneqq (u, \operatorname{tr}_{\Fh} (\dn u)),
\end{align*}
so that $\hdgjump{\dn u} = 0$ by construction. Setting $\Delta_h \u = \Delta u$, the discrete norm $\Vert \cdot \Vert_{V_h} $ is well-defined for such tupless $\u \in V \times L^2(\Fh)$ with $\Vert \u \Vert_{V_h} = \Vert u \Vert_{V}$, and the vanishing jump implies that $\sterm_h(\u,\vh) = 0$ for all $\vh \in V_h$.

Next, we introduce a suitable interpolation operator $\Pih : V \cap H^{2+s}(\Omega) \to V_h$, $s \ge 0$. For each element $T \in \Th$, we define
\begin{align}
   \Pih \vert_{T} : V \cap H^{2+s}(T) \to \mathbb{P}^p(T) \times \mathbb{P}^{p-1}(\partial T), \quad u \mapsto (\Pi_h^p u, \Pi_F^{p-1} (\dn u)),
\end{align}
where $\Pi_h^p$ is the Lagrange interpolation operator and $\Pi_F^{p-1}$ is the $L^2$-orthogonal projection onto $\mathbb{P}^{p-1}(F)$, $F \in \Fh$. The Lagrange interpolation operator is well-defined on $V$, since $H^2(\Omega) \hookrightarrow C^0(\overline{\Omega})$ for $d \in \{2,3\}$. 

\begin{lemma}\label[lemma]{lem:PihApprox}
   For all $u \in V \cap H^{2+s}(\Omega)$, $s \ge 0$, we have that 
   \begin{align*}
      \Vert \u - \Pih u \Vert_{V_h} \le C h^{\min\{p-1,s\}} \Vert u \Vert_{H^{2+s}(\Omega)}.
   \end{align*}
\end{lemma}

\begin{proof}
   Standard interpolation results \cite[Thm.~11.13]{EG_FE1} for the Lagrange interpolation operator yield that 
   \begin{align}\label{eq:LagrInterpolEst}
      \vert u - \Pi_h^p u \vert_{H^m(T)} \le C h^{r-m} \vert u \vert_{H^r(T)}, \qquad 2 \le r \le p+1, \ m \le r.
   \end{align}
   Since $\hdgjump{\dn \u} = 0$, we have that 
   \begin{align*}
      \Vert \u - \Pih u \Vert_{V_h} \le C \left( \Vert u - \Pi_h^p u \Vert_{H^2(\Th)} + \vert \Pih u \vert_{\sterm_h} \right).
   \end{align*}
   For the second term, observe that $\dn \Pi_h^p u \vert_F \in \mathbb{P}^{p-1}(F)$ for all $F \in \Fh$, so that  
   \begin{align*}
      \vert \Pih u \vert_{\sterm_h}  = \Vert h^{-1/2} \Pi_F^{p-1} \dn (u - \Pi_h^p u) \Vert_{\partial \Th} \le C \left( h^{-1} \vert u - \Pi_h^p u \vert_{H^1(\Th)} + \vert u - \Pi_h^p u \vert_{H^2(\Th)} \right),
   \end{align*}
   where we used the multiplicative trace inequality \cite[Lem.~12.15]{EG_FE1} in the last step.
   Applying \eqref{eq:LagrInterpolEst} with $r = \min\{2+s,\, p+1\}$ yields the claim.
\end{proof}

As a candidate for the projection operator, we define $\pih : V \to V_h$ by
\begin{align*}
   (\pih u,\vh)_{V_h} =  \gamma^{-1} (\Delta u, \Delta_h \vh)_{\Th} + (u,\vvol)_{H^1(\Omega)} \quad \forall \vh \in V_h
\end{align*}
and denote by $\pi_h$ the projection onto the volume component of $\pih$. Clearly, the operators $\pih$ are uniformly bounded. From the following result it directly follows that $\pih$ indeed satisfies \eqref{eq:pihCondition}.

\begin{lemma}\label[lemma]{lem:pihtoZero}
   For all $u \in V$, we have that $\lim_{h \to 0} \Vert \u - \pih u \Vert_{V_h} = 0$. 
\end{lemma}

\begin{proof}
   By definition of $\pih$, we have that 
   \begin{align*}
      \Vert \u - \pih u \Vert_{V_h}^2 = (\u - \pih u,\u - \Pih u)_{V_h} \le \Vert \u - \pih u \Vert_{V_h} \Vert \u - \Pih u \Vert_{V_h}.
   \end{align*}   
   Hence, \Cref{lem:PihApprox} implies that 
   \begin{align}\label{eq:pihtoZeroWithSgezero}
      \lim_{h \to 0} \Vert \u - \pih u \Vert_{V_h} = 0 \quad \forall u \in V \cap H^{2+s}(\Omega), s > 0.
   \end{align}
   For any $u \in V$, we can find by density $\tilde{u} \in C^\infty(\Omega) \cap V$ such that $\Vert u - \tilde{u} \Vert_{V} < \epsilon$ for any $\epsilon > 0$. Then, we have that 
   \begin{align*}
      \Vert \u - \pih u \Vert_{V_h} \le \Vert u - \tilde{u} \Vert_{V} + \Vert \ul{\tilde{u}} - \pih \tilde{u} \Vert_{V_h} + \Vert \pih (\tilde{u} - u) \Vert_{V_h} \le 2 \epsilon +  \Vert \ul{\tilde{u}} - \pih \tilde{u} \Vert_{V_h}.
   \end{align*} 
   The last term vanishes as $h \to 0$ by \eqref{eq:pihtoZeroWithSgezero}, so that $\limsup_{h \to 0} \Vert \u - \pih u \Vert_{V_h} \le 2\epsilon$. Since $\epsilon > 0$ was arbitrary, the claim follows.
\end{proof}

\section{Analysis of the discrete problem}\label{sec:NumAnalysis}
The main goal of this section is to show that the discrete problem \eqref{eq:discrete_weak_form} is stable under the same restrictions on $\Omega$, $\alpha$, and $\beta$ as in \Cref{subsec:WP} and that the solution $\uh \in V_h$ converges to the continuous solution $u \in V$ of \eqref{eq:Cont:WF}. 

After some preparations in \Cref{subsec:NA:prelims}, we show in \Cref{subsec:NA:AsympCons} that the discrete sesquilinear form is asymptotically consistent in the sense of \Cref{def:asympConsistency}. Then, in \Cref{sec:DA:stability}, we establish stability and convergence by applying \Cref{thm:discreteGarding}, exploiting the fact that $V_{\Th} \subset H^1(\Omega)$. Finally, we derive convergence estimates in \Cref{subsec:NA:ConvEst} under the assumption that the continuous solution $u$ has additional regularity. 

\subsection{Preliminaries}\label{subsec:NA:prelims}
At the discrete level, we need to control the piecewise Hessian. Since $V_{\Th} \not \subset H^2(\Omega)$, \Cref{lem:ellipticReg} does not apply directly. Instead, we prove a discrete analogue by using an enrichment operator $E_h : V_{\Th} \to H^2(\Omega)$ that maps the $H^1$-conforming Lagrange finite element space $V_{\Th}$ into an $H^2$-conforming space. 

Classically, such operators map $V_{\Th}$ into a $C^1$-conforming counterpart, for instance, the Argyris space or the Hsieh--Clough--Tocher macro-element \cite{BS05,BGS10,GHV11}, but in the three-dimensional case these constructions are limited to polynomial degrees $p \leq 3$.
This drawback is overcome in \cite{BS19VEO} by mapping into $H^2$-conforming virtual elements, which are available for any polynomial degree $p \ge 2$. While the enrichment operator $E_h$ is also useful for designing fast solvers for fourth-order problems \cite{BW05}, we use it here only as a theoretical tool and refer to the literature for details on its construction \cite{BS05,BGS10,GHV11,BS19VEO}.

\begin{lemma}\label[lemma]{lem:EnrichmentOperator}
   Let $p \ge 2$. There exists an operator $E_h : V_{\Th} \to H^2(\Omega)$ such that 
   \begin{align*}
      \Vert \hessian \left( \vvol - E_h \vvol \right) \Vert_{\Th} \le C \Vert h^{-1/2} \jump{\dn \vvol} \Vert_{\Fh^{i}},
   \end{align*}
   where the constant $C > 0$ depends only on $p$ and the shape regularity of $\Th$. In particular, $E_h$ maps $V_{\Th}\cap H^1_0(\Omega)$ into $H^2(\Omega) \cap H^1_0(\Omega)$.
\end{lemma}

\begin{proof}
   For $p = 2$, we can use the standard constructions from \cite{BS05,BGS10,GHV11}. For general $p \ge 2$, we refer to \cite[Sec.~2.4]{BS19VEO} for $d = 2$ and to \cite[Sec.~3.4]{BS19VEO} for $d = 3$.
\end{proof}

The enrichment operator allows us to prove the following discrete analogue of \Cref{lem:ellipticReg}.

\begin{lemma}\label[lemma]{lem:discreteHessianEst}
   Let $h \le 1$. For any $\uh \in V_h$, we have that 
   \begin{align}\label{eq:discreteHessianEst}
      \Vert \hessian \uvol \Vert_{\Th} \le \CR \left( \Vert \Delta_h \uh \Vert_{\Th} + \Vert \uvol \Vert_{\Omega} \right) + \CE \vert \uh \vert_{\sterm_h},
   \end{align}
   where $\CR > 0$ is the elliptic regularity constant from \eqref{eq:ellipticReg} and $\CE > 0$ is independent of $h$.
\end{lemma}
 
\begin{proof}
   Since $E_h \uvol \in H^2(\Omega)$, we can apply \Cref{lem:ellipticReg} to estimate 
   \begin{align*}
      \Vert \hessian E_h \uvol \Vert_{\Omega} \le \CR \left( \Vert \Delta E_h \uvol \Vert_{\Omega} + \Vert E_h \uvol \Vert_{\Omega} \right).
   \end{align*}
   Combining this with the triangle inequality and the elementwise estimate $\Vert \Delta v \Vert_{\Th} \le \sqrt{d} \Vert \hessian v \Vert_{\Th}$, we obtain
   \begin{align*}
      \Vert \hessian \uvol \Vert_{\Th} &\le \CR \left( \Vert \Delta E_h \uvol \Vert_{\Omega} + \Vert E_h \uvol \Vert_{\Omega} \right) + \Vert \hessian (\uvol - E_h \uvol) \Vert_{\Th} \\
      &\le \CR \left( \Vert \Delta \uvol \Vert_{\Th} + \Vert \uvol \Vert_{\Omega} \right) + (1 + \sqrt{d} \CR) \Vert \hessian (\uvol - E_h \uvol) \Vert_{\Th} \\
      &\qquad+ \CR \Vert \uvol - E_h \uvol \Vert_{\Omega}.
   \end{align*}
   It remains to bound the two terms involving $\uvol - E_h \uvol$ by the stabilization seminorm. \Cref{lem:EnrichmentOperator} yields that  
   \begin{align*}
      \Vert \hessian (\uvol - E_h \uvol) \Vert_{\Th} \le C \Vert h^{-1/2} \jump{\dn \uvol} \Vert_{\Fh^{i}} \le C \vert \uh \vert_{\sterm_h},
   \end{align*} 
   and furthermore, the construction from \cite[Eq.~(2.30)]{BS19VEO} yields that 
   \begin{align}
      \Vert \uvol - E_h \uvol \Vert_{\Omega} \le C h^{3/2} \Vert h^{-1/2} \jump{\dn \uvol} \Vert_{\Fh^{i}} \stackrel{h \le 1}{\le} C \vert \uh \vert_{\sterm_h}.
   \end{align}
   Since $\Vert \Delta \uvol \Vert_{\Th} \le \Vert \Delta_h \uh \Vert_{\Th} + \Clift \vert \uh \vert_{\sterm_h}$, combining all three estimates yields the claim.
\end{proof}

\begin{remark}[Lehrenfeld--Sch\"{o}berl stabilization] 
   In the context of HDG methods, a natural idea would be to replace the stabilization term $\sterm_h(\cdot,\cdot)$ in the sesquilinear form and the norm by the Lehrenfeld--Sch\"{o}berl stabilization term $\sterm_h(\lsp \cdot, \lsp \cdot)$, allowing us to reduce the polynomial degree of the facet space $V_{\Fh}$ to $p-2$ \cite{L10,LS16}. Unfortunately, we cannot\footnotemark \ weaken the stabilization term in \eqref{eq:discreteHessianEst}. Following \cite[Lem.~4.1]{DE22}, we can estimate 
   \begin{align*}
      \vert \vh \vert_{\sterm_h} \le \vert \lsp \vh \vert_{\sterm_h} + \vert \lspPerp \hdgjump{\dn v_h}\vert_{\sterm_h} \le \vert \lsp \vh\vert_{\sterm_h} + C \Vert \hessian \vvol \Vert_{\Th},
   \end{align*}
   but the Hessian term on the right-hand side cannot be controlled with the sesquilinear form $a_h(\cdot,\cdot)$. Hence, at least with the current approach, we cannot show stability of the proposed method with the Lehrenfeld--Sch\"{o}berl stabilization term. The only exception is the case $\beta = 0$, where the estimate \eqref{eq:discreteHessianEst} can be avoided.
   \footnotetext{On mesh consisting of two elements, choose a harmonic function $\uvol \in \mathbb{P}^2(\Th)$ whose normal-derivative jump across the interface has zero mean value. With the weaker stabilization $\sterm_h(\lsp \cdot, \lsp \cdot)$, the right-hand side of \eqref{eq:discreteHessianEst} contains only the $H^1$-term, which does not control the full Hessian of $u_h$ with a constant independent of $h$.}
\end{remark}

A direct consequence of \Cref{lem:discreteHessianEst} and the Cauchy--Schwarz inequality is the continuity of $a_h(\cdot,\cdot)$ with respect to the $\Vert \cdot \Vert_{V_h}$-norm. 

\begin{lemma}[Continuity]
   There exists a constant $C > 0$ such that
   \begin{align*}
      \left\vert a_h(\uh,\vh) \right\vert \le C \Vert \uh \Vert_{V_h} \Vert \vh \Vert_{V_h} \qquad \forall \uh,\vh \in V_h.
   \end{align*}
\end{lemma}

\subsection{Asymptotic consistency}\label{subsec:NA:AsympCons}
The main goal of this section is to show that the sesquilinear form $a_h(\cdot,\cdot)$ is asymptotically consistent with $a(\cdot,\cdot)$ in the sense of \Cref{def:asympConsistency}. In preparation, we prove a discrete compactness result, which we use both in the proof of asymptotic consistency and in the verification of \Cref{ass:InclusionConsistentlyCompact}.  

\begin{theorem}\label{thm:discreteCompactnessResult}
   Let $\seqh{\ul{v}}$, $\vh \in V_h$, be uniformly bounded. Then, there exists $v \in V$ such that, up to subsequences, $\vvol \overset{L^2}{\rightharpoonup} v$, $\nabla \vvol \overset{[L^2]^d}{\rightharpoonup} \nabla v$, and $\Delta_h \vh \overset{L^2}{\rightharpoonup} \Delta v$. In particular, $\vh \overset{\pi_h}{\rightharpoonup} v$.
\end{theorem}

\begin{proof}
   The proof follows standard arguments \cite{BE08,BO09,DPE10,HLvB25}. Since the volume unknowns $\vvol \in V_{\Th} \subset H^1(\Omega)$ form a bounded sequence in $H^1(\Omega)$, there exists $v \in H^1(\Omega)$ such that $\vvol \overset{L^2}{\rightharpoonup} v$ and $\nabla \vvol \overset{[L^2]^d}{\rightharpoonup} \nabla v$. Further, $ \Vert \Delta_h \vh\Vert_{\Th} \le \gamma^{1/2} \Vert \vh \Vert_{V_h}$, so that the sequence $\seqh{\Delta_h \ul{v}}$ is uniformly bounded in $L^2$ and there exist $q \in L^2(\Omega)$ such that $\Delta_h \vh\overset{L^2}{\rightharpoonup} q$, up to subsequences. It remains to show that $q = \Delta v$. Let $\psi \in C_0^\infty(\Omega)$ and let $\psi_h$ be the element-wise $L^2$-projection of $\psi$ onto $\mathbb{P}^{p-2}(\Th)$. Since $\vvol \in H^1(\Omega)$, integration by parts yields that
   \begin{subequations}\label{eq:dCP:partialInt}
      \begin{align}
         (\vvol, \Delta \psi)_{\Th} &= -(\nabla \vvol, \nabla \psi)_{\Th} = (\Delta \vvol, \psi)_{\Th} - (\hdgjump{\dn v_h},\psi)_{\partial \Th} \\ 
          &= (\Delta \vvol, \psi - \psi_h)_{\Th} - (\hdgjump{\dn v_h},\psi - \psi_h)_{\partial \Th} + (\Delta_h \vh,\psi_h)_{\Th} \label{eq:dCP:partialInt:b}, 
       \end{align}
   \end{subequations}   
   where, in the second step, we used that $\psi \in H^1(\Omega)$ to shift in the facet component in the second step. Since $\lim_{h \to 0} \Vert \psi - \psi_h \Vert_{\Th} = 0$, the first two terms in \eqref{eq:dCP:partialInt:b} go to zero as $h \to 0$, and it follows that 
   \begin{align*}
      (q,\psi)_{\Th} = \lim_{h \to 0} (\Delta_h \vh, \psi)_{\Th} = \lim_{h \to 0} ((\Delta_h \vh, \psi - \psi_h)_{\Th} + (\Delta_h \vh, \psi_h)_{\Th}) \stackrel{\eqref{eq:dCP:partialInt}}{=} \lim_{h \to 0} (\vvol, \Delta \psi)_{\Th} = (v, \Delta \psi)_{\Omega}.
   \end{align*}
   Hence, we have that $q = \Delta v \in L^2(\Omega)$ in the sense of distributions, and \eqref{eq:ellipticReg} yields that $v \in V$.
\end{proof}

\begin{lemma}\label[lemma]{lem:asymptoticConsistencyah}
   The sesquilinear form $a_h(\cdot,\cdot)$ is asymptotically consistent with $a(\cdot,\cdot)$.
\end{lemma}

\begin{proof}
   Let $u \in V$ and let $\seqh{\ul{v}}$, $\vh \in V_h$, be a uniformly bounded sequence. By \Cref{thm:discreteCompactnessResult}, there exists $v \in V$ such that $\vvol \overset{L^2}{\rightharpoonup} v$, $\nabla \vvol \overset{L^2}{\rightharpoonup} \nabla v$, and $\Delta_h \vh\overset{L^2}{\rightharpoonup} \Delta v$ up to subsequences. We treat the terms of $a_h(\pih u, \vh)$ separately. For the boundary term, we estimate estimate that 
   \begin{align*}
      \vert (\pi_h u, \vvol)_{\Fh^R} \vert\le \vert  (\pi_h u - u, \vvol)_{\Fh^R} \vert  + (u, \vvol)_{\Fh^R}  \le C \Vert \pi_h u - u \Vert_{H^1(\Omega)} \Vert v_h \Vert_{V_h} + (u, \vvol)_{\Fh^R}, 
   \end{align*}
   where we used that $\pi_h u- u \in H^1(\Omega)$. Since $\vh$ is a bounded sequence, the first term converges to zero due to \Cref{lem:pihtoZero} and the second term converges to $(u,v)_{\GammaR}$ due to the compactness of the trace operator.
   For the volume terms, we use the same technique and estimate 
   \begin{align*}
      \vert (\FOterm{\pi_h u},\Delta_h \vh)_{\Th} + \stab \sterm_h(\pih u, \vh) \vert  \le C \Vert \pih u - \u \Vert_{V_h} \Vert \vh \Vert_{V_h} + (\FOtermu, \Delta_h \vh)_{\Th}.
   \end{align*}
   The first term on the right-hand side converges to zero by \Cref{lem:pihtoZero} and the second term converges to $(\FOtermu,\Delta v)_{\Omega}$ since $\Delta_h \vh \overset{L^2}{\rightharpoonup} \Delta v$. The lower order volume terms $(\nabla \pi_h u, \nabla \vvol)_{\Omega} - k^2(\pi_h u,\vvol)_{\Omega}$ can be treated analogously. Altogether, we obtain that 
   \begin{align*}
      \lim_{h \to 0} a_h(\pih u,\vh) = a(u,v),
   \end{align*}
   which shows that $a_h(\cdot,\cdot)$ is asymptotically consistent with $a(\cdot,\cdot)$.
\end{proof}

Finally, we show that the space $V_h = V_{\Th} \times V_{\Fh}$ defined by \eqref{eq:def:VhSpaces} satisfies \Cref{ass:InclusionConsistentlyCompact}.

\begin{lemma}\label[lemma]{lem:discreteIota}
   The sequence of discrete spaces $\seqh{V}$ satisfies \Cref{ass:InclusionConsistentlyCompact} with $W = L^2(\Omega)$. 
\end{lemma}

\begin{proof}
   Let $\Pvol : V_h \to V_{\Th}$, $\uh = (u_h,\ufac) \mapsto u_h \in V_{\Th}$ denote the projection onto the volume component. We define $\iota_h \coloneqq \iota_{H^1 \to L^2} \circ \Pvol: V_h \to L^2$, where $\iota_{H^1 \to L^2}$ is the compact embedding map $H^1 \hookrightarrow L^2$. Since $\Vert \iota_h \uh \Vert_{\Omega} \le \Vert \Pvol \uh \Vert_{H^1(\Omega)} \le \Vert \uh \Vert_{V_h}$, the operators $\iota_h$ are uniformly bounded. Define $J \coloneqq \iota_{H^1 \to L^2} \circ \iota_{H^2 \to H^1} : V \to L^2(\Omega)$. For any uniformly bounded sequence $\seqh{\ul{v}}$, \Cref{thm:discreteCompactnessResult} provides $v \in V$ such that $\Pvol \vh \overset{H^1}{\rightharpoonup} \iota_{H^2 \to H^1} v$ up to subsequences. Since $\iota_{H^1 \to L^2}$ is compact, we thus obtain the strong convergence $$\iota_h \vh = \iota_{H^1 \to L^2} \Pvol \vh \to \iota_{H^1 \to L^2} \iota_{H^2 \to H^1} v = J v$$ in $L^2(\Omega)$, which yields the claim.
 \end{proof}

\subsection{Stability}\label{sec:DA:stability}
We now apply \Cref{thm:discreteGarding} to establish the stability of the discrete problem \eqref{eq:discrete_weak_form}. To this end, we show that, under the same assumptions as in \Cref{lem:weakcoercivity}, the discrete sesquilinear form satisfies \eqref{eq:Abstract:DiscrGarding} provided that the stabilization parameter $\stab > 0$ is sufficiently large. In particular, the restrictions on $\alpha$ and $\beta$ from \Cref{lem:alphabetaConditionImpliesCordes} also apply in the discrete setting.

\begin{theorem}\label{thm:uniformlyCoercive}
   Suppose that $\CR \vert \gamma \Am - \Idm \vert_{F} < 1$ and that $\stab > 0$ is sufficiently large. Then, there exist constants $\tilde{C}_V, \tilde{C}_{L^2} > 0$ independent of $h$ such that 
   \begin{align*}
      \Re a_h(\uh,\uh) \ge \tilde{C}_{V} \Vert \uh \Vert^2_{V_h} - \tilde{C}_{L^2} \Vert \uvol \Vert^2_{\Omega}, \qquad \forall \uh \in V_h.
   \end{align*}
\end{theorem}

\begin{proof}
   We adapt the proof of \Cref{lem:weakcoercivity} to the discrete setting; the main difference is that the discrete Hessian estimate from \Cref{lem:discreteHessianEst} replaces its continuous counterpart. Throughout, we abbreviate $\dCr \coloneqq \CR \vert \gamma \Am - \Idm \vert_{F}$ and $\dCe \coloneqq \CE \vert \gamma \Am - \Idm \vert_{F}$ and note that $\dCr < 1$ by assumption.
   Starting from the rewriting \eqref{eq:ahLift} of the discrete sesquilinear form, we obtain for any $\epsilon_1 > 0$ that
   \begin{align*}
      \left\vert (\gamma \FOtermuh - \Delta \uvol, \Delta_h \uh)_{\Th} \right\vert &\overset{\eqref{eq:cordesEstimate}}{\le} \vert \gamma \Am - \Idm \vert_{F} \Vert \hessian \uvol \Vert_{\Th} \Vert \Delta_h \uh \Vert_{\Th} \\
      &\overset{\eqref{eq:discreteHessianEst}}{\le} \dCr (1 + 2 \epsilon_1) \Vert \Delta_h \uh \Vert^2_{\Th} + \frac{\dCr}{4 \epsilon_1} \Vert \uvol \Vert^2_{\Omega} + \frac{\dCe^2}{4 \dCr \epsilon_1} \vert \uh \vert^2_{\sterm_h},
   \end{align*}
   where the parameter $\epsilon_1$ stems from a weighted Young's inequality. Similarly, \eqref{eq:LiftBounded} yields for any $\epsilon_2 > 0$ that
   \begin{align}
      \Re (\Delta \uvol, \Delta_h \uh)_{\Th} \ge  (1 - \epsilon_2) \Vert \Delta_h \uh \Vert^2_{\Th} - \frac{\Clift^2}{4 \epsilon_2} \vert \uh \vert^2_{\sterm_h}.
   \end{align}
   Combining the two estimates, we obtain the discrete counterpart of \eqref{eq:ContCoercivity:AD2Estimate}:
   \begin{align*}
      \Re (\FOtermuh,\Delta_h \uh)_{\Th} \ge &\frac{1}{\gamma}(1-\dCr(1 + 2 \epsilon_1) - \epsilon_2) \Vert \Delta_h \uh \Vert^2_{\Th} - \frac{\dCr}{4 \gamma \epsilon_1} \Vert \uvol \Vert^2_{\Omega} \\
      &- \gamma^{-1} \left( \frac{\dCe^2}{4 \dCr \epsilon_1} + \frac{\Clift^2}{4 \epsilon_2} \right) \vert \uh \vert^2_{\sterm_h}.
   \end{align*}
   Since $\dCr < 1$, we can choose $\epsilon_1, \epsilon_2 > 0$ so small that $1 - \dCr(1 + 2\epsilon_1) - \epsilon_2 > 0$. If, in addition, $\stab$ is large enough that $\stab \gamma > \dCe^2 (4 \dCr \epsilon_1)^{-1} + \Clift^2 (4 \epsilon_2)^{-1}$, then
   \begin{align*}
      \Re a_h(\uh,\uh) &\ge \frac{C}{\gamma} \Vert \Delta_h \uh \Vert^2_{\Th}  + \Vert \nabla \uvol \Vert^2_{\Omega} + \Vert \uvol \Vert^2_{\Omega} + (\stab \gamma - \dCe^2 (4 \dCr \epsilon_1)^{-1} - \Clift^2(4 \epsilon_2)^{-1}) \gamma^{-1} \vert \uh \vert^2_{\sterm_h} \\
      &\quad - (1 + k^2 + \dCr(4 \gamma \epsilon_1)^{-1}) \Vert \uvol \Vert^2_{\Omega} \\
      &\ge \tilde{C}_{V} \Vert \uh \Vert^2_{V_h} - \tilde{C}_{L^2} \Vert \uvol \Vert^2_{\Omega},
   \end{align*}
   which proves the claim.
\end{proof}

Thus, the discrete sesquilinear form satisfies \eqref{eq:Abstract:DiscrGarding} which, together with \Cref{lem:discreteIota}, allows us to apply \Cref{thm:discreteGarding} to obtain the following result. 

\begin{theorem}[Stability \& convergence]\label{thm:Stability}
   Suppose that $\beta$ satisfies \eqref{eq:alphabetaCondition} and that \Cref{ass:injectivity} holds.
   If $\stab > 0$ is sufficiently large, then there exists $h_0 > 0$ such that for all $h \le h_0$ the discrete problem \eqref{eq:discrete_weak_form} has a unique solution $\uh \in V_h$. Furthermore, if $u \in V$ solves \eqref{eq:Cont:WF}, we have that 
   \begin{align}
      \lim_{h \to 0} \Vert \u - \uh \Vert_{V_h} = 0.
   \end{align}
\end{theorem}

\begin{proof}
   By \Cref{lem:asymptoticConsistencyah}, $a_h(\cdot,\cdot)$ is asymptotically consistent with $a(\cdot,\cdot)$, and $\pih$ satisfies \eqref{eq:pihCondition} by \Cref{lem:pihtoZero}. Thus $(a_h,\pih,V_h)$ is a DAS of $(a,V)$. Moreover, $a_h(\cdot,\cdot)$ satisfies \eqref{eq:Abstract:DiscrGarding} by \Cref{thm:uniformlyCoercive}, \Cref{ass:InclusionConsistentlyCompact} holds by \Cref{lem:discreteIota}, and $a(\cdot,\cdot)$ satisfies \eqref{eq:Abstract:Garding} by \Cref{lem:weakcoercivity}. Hence, under \Cref{ass:injectivity}, \Cref{thm:discreteGarding} yields the existence of $h_0 > 0$ such that $a_h(\cdot,\cdot)$ is stable, and \eqref{eq:discrete_weak_form} is uniquely solvable for all $h \le h_0$.

   It remains to prove convergence. For $h \le h_0$, let $\uh \in V_h$, $h \le h_0$ be the unique solution of \eqref{eq:discrete_weak_form}. Then, the triangle inequality gives that 
   \begin{align}\label{eq:convergenceTriangle}
      \Vert \u - \uh \Vert_{V_h} \le \Vert \u - \pih u \Vert_{V_h} + \Vert \pih u - \uh \Vert_{V_h}.
   \end{align}
   The first term converges to zero by \Cref{lem:pihtoZero}. For the second term, stability implies that 
   \begin{align}\label{eq:convergenceStab}
      \Vert \pih u - \uh \Vert_{V_h} \le \sconst^{-1} \! \! \! \! \! \!\! \! \! \sup_{\wh \in V_h, \Vert \wh \Vert_{V_h} = 1} \vert a_h(\pih u - \uh, \wh) \vert.
   \end{align} 
   Since the sequence $\seqh{\ul{w}}$, $\wh \in V_h$, is uniformly bounded, \Cref{thm:discreteCompactnessResult} and the asymptotic consistency of $a_h(\cdot,\cdot)$ imply that there exists $w \in V$ such that $\wvol \overset{L^2}{\rightharpoonup} w$ and (up to subsequences) we have that 
   \begin{align*}
      a_h(\pih u - \uh, \wh) = a_h(\pih u, \wh) - (f,\wvol)_{\Omega} \to a(u,w) - (f,w)_{\Omega} = 0,
   \end{align*}
   where we used that $\uh$ solves \eqref{eq:discrete_weak_form} and $u$ solves \eqref{eq:Cont:WF}.
   Together with \eqref{eq:convergenceTriangle}, this proves the claim. 
\end{proof}

\subsection{Convergence estimate}\label{subsec:NA:ConvEst}
We close the analysis of the discrete problem with an estimate for the rate of convergence in \Cref{thm:Stability}, assuming that the exact solution $u \in V$ has additional regularity.

\begin{theorem}[Convergence estimate]\label{thm:convergenceEst}
   Suppose that $\beta$ satisfies \eqref{eq:alphabetaCondition} and that \Cref{ass:injectivity} holds.
   Let $u \in V \cap H^{2+s}(\Omega)$, $s > 1$, be the solution to \eqref{eq:Cont:WF}, let $\stab  > 0$ be sufficiently large, and let $h_0 > 0$ such that the discrete problem \eqref{eq:discrete_weak_form} has a unique solution $\uh \in V_h$ for all $h \le h_0$. Then, there exists a constant $C > 0$ independent of $h$ such that
   \begin{align}
      \Vert \u - \uh \Vert_{V_h} \le C h^{\min\{p-1,s\}} \Vert u \Vert_{H^{2+s}(\Omega)}.
   \end{align}
\end{theorem}

\begin{proof}
   Starting from \eqref{eq:convergenceTriangle}, we only have to estimate the discrete error $\Vert \pih u - \uh \Vert_{V_h}$ because we obtain the desired rate of convergence for the first term with the arguments from the proof of \Cref{lem:pihtoZero}.
   The regularity assumption on $u$ ensures that $\Delta u, \nHn{u} \in H^1(\Omega)$ and $\Delta^2 u, \nabla \cdot \nabla \nHn{u} \in H^{-1}(\Omega)$, so that $u$ can be inserted into the discrete sesquilinear form $a_h(\cdot,\cdot)$. Using integration by parts (in reverse) and removing the facet variable yields that
   \begin{align*}
      a_h(\u,\vh) &= \langle \Delta (\FOtermu), \vvol \rangle_{H^{-1},H^1} - (\Delta u,\vvol)_{\Omega} - k^2 (u,\vvol)_{\Omega} = (f,\vvol)_{\Omega} = a_h(\uh,\vh),
   \end{align*}
   that is, the sesquilinear form is consistent in the classical sense: $ a_h(\u - \uh,\vh) = 0$ for all $\vh \in V_h$.
   Combining stability as in \eqref{eq:convergenceStab}, this Galerkin orthogonality, and the continuity of $a_h(\cdot,\cdot)$ yields that
   \begin{align*}
      \sconst \Vert \pih u - \uh \Vert_{V_h} \le \! \sup_{\wh \in V_h} \frac{\vert a_h(\pih u - \uh, \wh) \vert}{\Vert \wh \Vert_{V_h}} = \! \! \sup_{\wh \in V_h} \frac{\vert a_h(\pih u - \u, \wh) \vert}{\Vert \wh \Vert_{V_h}} \le C_{\text{{cont}}}\Vert u - \pih u \Vert_{V_h}.
   \end{align*}
   A second application of \Cref{lem:PihApprox} concludes the proof.
\end{proof}

\section{Numerical experiments}\label{sec:numex}
In this section, we present numerical examples. First, we verify the convergence rates predicted by \Cref{thm:convergenceEst} numerically and further investigate the stability of the discrete problem. Afterwards, we showcase the anisotropic effect of the nematic field in two and three dimensions. 
All experiments are implemented using the finite element software \texttt{Netgen/NGSolve} \cite{Sch97,Sch14}. 
Replication data are available in \cite{ReplicationData}.
For readability, we leave out renormalization factors in front of the nematic fields $\bm{n}$.

\subsection{Convergence study \& stability}
We consider the unit disk $\Omega = \{ \bm{x} \in \mathbb{R}^2 : \Vert \bm{x} \Vert_{2} \le 1 \} \subset \mathbb{R}^2$ and set sound-soft boundary conditions everywhere, that is, $\GammaD = \partial \Omega$. 
As an exact solution, we consider a plane wave of the form
\begin{equation}
   u_{\text{ex}}(\bm{x}) = e^{i(\bm{k}\cdot\bm{x})}, \qquad \bm{k}\in \mathbb{R}^d \label{eq:planewave}, 
\end{equation}
where the wave-vector $\bm{k}$ satisfies the following dispersion relation: 
\begin{equation}
   \alpha \kappa^4 + \beta \kappa^2 (\bm{k}^{T}\bm{n})^2+ \kappa^2 - k^2 = 0, \qquad \kappa = \vert \bm{k} \vert. \label{eq:dispersion}
\end{equation}
Then, \eqref{eq:planewave} is a solution to \eqref{eq:nematic_helmholtz_korteweg} with right-hand side zero \cite{FZ24,FvBZ25}. We homogenize the problem to obtain sound-soft boundary conditions.

Since we consider pure sound-soft boundary conditions and $\Omega \subset \mathbb{R}^2$ is convex, the discrete problem is stable for all choices of $\alpha > 0$ and $\beta > - \alpha$. Moreover, since \eqref{eq:planewave} is sufficiently smooth, we expect convergence of order $\mathcal{O}(h^{p-1})$ by \Cref{thm:convergenceEst}.  
In \Cref{fig:numex:convergenceCircle}, we compare the errors in the broken $\Vert \cdot \Vert_{V}$-norm for polynomial degrees $p \in \{ 2,3,4,5\}$. We fix the nematic field $\bm{n} = (1,0)^{\top}$, the wave number $k = 30$, the stabilization parameter $\stab = 50 p^2$, and the parameter $\alpha = 10^{-2}$. Then, we consider the cases $\beta \in \{-10^{-3}, 0, 10^{-3}, 1\}$. In most cases, we observe that the errors decay with the predicted rate of $\mathcal{O}(h^{p-1})$. Exceptions include minor fluctuations for $p = 2$ and a slight error increase in the final refinement step with $\beta = 1$ and $p = 5$. These deviations may be attributed to the generic choice of $\stab$ across all values of $\beta$; a better choice of the parameter $\stab$ would depend on $\gamma$ and $\vert \gamma \Am - \Idm \vert_{F}$.

\begin{figure}[!htbp]
  \centering
  \includegraphics[width=0.95\textwidth]{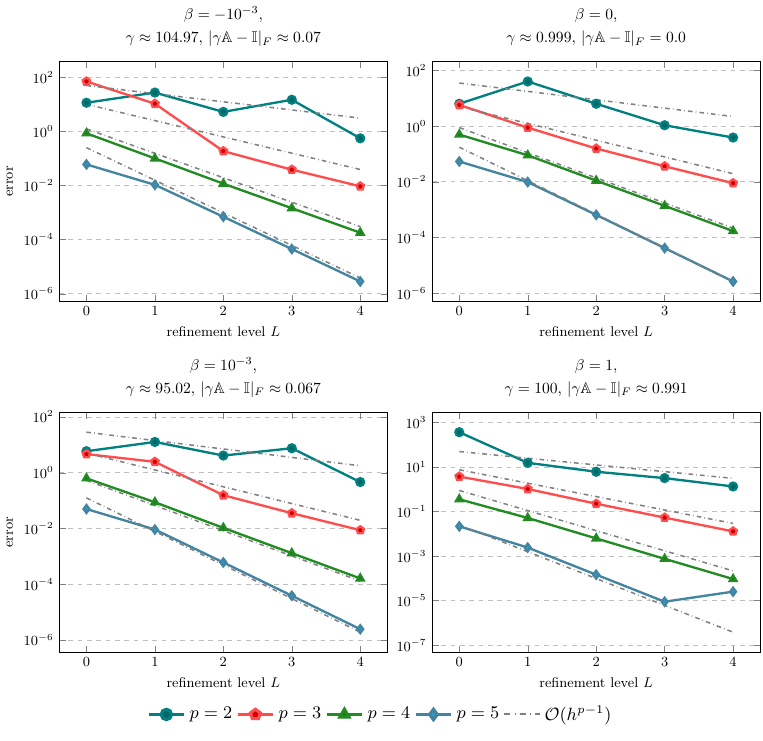}
  \caption{We consider the nematic director $\bm{n} = (1,0)^{\top}$, the wave number $k = 30$, stabilization parameter $\stab = 50 p^2$, and fix the parameter $\alpha = 10^{-2}$. We compare the cases $\beta = -10^{-3}$ (top left), $\beta = 0$ (top right), $\beta = 10^{-3}$ (bottom left), and $\beta = 1$ (bottom right). We compute the error in the broken $\Vert \cdot \Vert_{V}$-norm on a sequence of uniformly refined meshes with $h \approx 0.2 \cdot 2^{-L}$, where $L \in \{0,1,2,3,4\}$ is the refinement level. For polynomial degrees $p \in \{2,3,4,5\}$, we observe convergence with order $\mathcal{O}(h^{p-1})$ as expected by \Cref{thm:convergenceEst}.}
  \label{fig:numex:convergenceCircle}
\end{figure}

To further investigate the stability of the method with respect to the smallness assumptions on $\alpha$ and $\beta$, we consider the unit ball $\Omega = \{ \bm{x} \in \mathbb{R}^d : \Vert \bm{x} \Vert \le 1 \}$ and set sound-soft boundary conditions $\GammaD = \partial \Omega$. \Cref{thm:Stability} together with \eqref{eq:alphabetaCondition} implies stability for $\beta \in (-\alpha, \infty) $ in the two-dimensional case and for $\beta \in (-\alpha, 3\alpha)$ in the three-dimensional case. In \Cref{fig:numex:stability}, we estimate the discrete inf-sup constant on a fixed mesh for $\beta \in (-1,4]$, $\alpha = 1$, $k = 5$, and $p = 2$. When $\beta$ approaches $-1$, the inf-sup constant seems to approach zero, i.e., the method is not stable. Curiously, the inf-sup constant remains positive in the three-dimensional case, even for $\beta  > 3 \alpha$, illustrating that the condition \eqref{eq:alphabetaCondition} is sufficient, but not necessary.

\begin{figure}[!htbp]
   \begin{center}
      \includegraphics[width=0.85\textwidth]{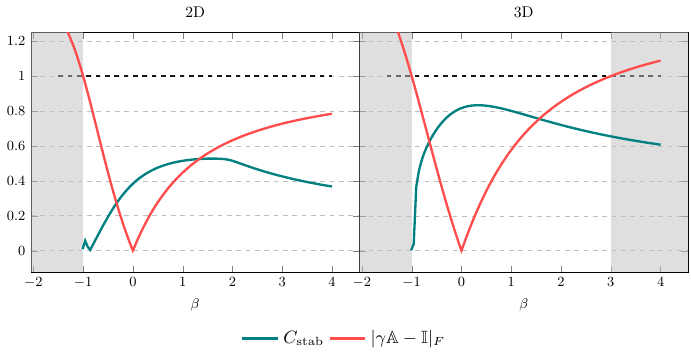}
   \end{center}
   \caption{On a fixed triangulation of the unit ball in $\mathbb{R}^d$, $d = 2,3$, we estimate the discrete inf-sup constant $C_{\text{stab}}$ for $\beta \in (-1,4]$ with $\alpha = 1$ and $k = 5$ fixed. We also plot $\vert \gamma \Am - \Idm \vert_{F}$ as a function of $\beta$ and indicate the regions (in gray) where the condition $\vert \gamma \Am - \Idm \vert_{F} < 1$ is violated. For $\beta \to -1$, the inf-sup constant seems to approach zero.}
   \label{fig:numex:stability}
\end{figure}

\begin{remark}[Computational aspects]
   The main motivation for hybridization of the $C^0$-interior penalty ($C^0$-IP) method is the reduction in computational costs due to static condensation.
   In \Cref{tab:numex:crunching}, we compare the degrees of freedom (\texttt{dofs}) and non-zero entries (\texttt{nzes}) of the system matrices for $C^0$-HIP and $C^0$-IP discretizations. The hybridization leads to significantly lower computational costs if the polynomial degree is large enough. This is in line with our expectations, since static condensation only eliminates higher-order interior bubbles. For a detailed comparison of both methods for the biharmonic problem, we refer to \cite{DE24C0HHO}.   
   \begin{table}[!htbp]
      \resizebox{0.5\textwidth}{!}{
      \begin{tabular}{c|cc|cc}%
         & \multicolumn{2}{c|}{$C^{0}$-hybrid IP}& \multicolumn{2}{c}{$C^{0}$-IP} \\ \hline
         $p$ & \texttt{dofs} & \texttt{nzes}  & \texttt{dofs} & \texttt{nzes}
         \begin{filecontents*}{standalones/numbercrunching.csv}
order,HDGndofs,HDGncdofs,HDGnzes,HDGne,HDGkappa,HDGkappaDiag,DGndofs,DGncdofs,DGnzes,DGne,DGkappa,DGkappaDiag,HDGnoLSPndofs,HDGnoLSPncdofs,HDGnoLSPnzes,HDGnoLSPne,HDGnoLSPkappa,HDGnoLSPkappaDiag
2,6257,3265,54881,944,188391.61,1547.01,1969,1809,46801,944,-24411.5,3831.66,4881,4721,101969,944,21829.85,1605.06
3,10113,6097,163297,944,44554.07,2450.95,4369,3185,171889,944,4231.29,1041.37,8737,7553,238865,944,7676.61,1832.13
5,20657,11761,551009,944,28676.3,2334.66,12001,5937,986401,944,6522.64,1517.97,19281,13217,683537,944,27463.11,3041.83
7,34977,17425,1166561,944,27056.37,2457.68,23409,8689,3293361,944,10359.16,2059.06,33601,18881,1356049,944,37064.5,3661.27
9,53073,23089,2009953,944,27343.88,2909.1,38593,11441,8291233,944,14150.16,2619.44,51697,24545,2256401,944,44946.52,4255.62
\end{filecontents*}
\csvreader[head to column names]{standalones/numbercrunching.csv}{}
         {\\ \rule{0pt}{2ex} \order & \num{\HDGnoLSPndofs} & \num{\HDGnoLSPnzes} & \num{\DGndofs} & \num{\DGnzes} } \\ \hline
      \end{tabular}
      }
      \vspace*{0.3cm}
      \caption{We compare the degrees of freedom (\texttt{dofs}) and the non-zero entries (\texttt{nzes}) of the system matrix of a $C^{0}$-HIP discretization with a $C^{0}$-IP discretization. The numbers are computed on the unit square with mesh size $h = 0.05$ (i.e., $944$ elements) for polynomial degrees $p \in \{2,3,5,7,9\}$. }
      \label{tab:numex:crunching}
   \end{table}
\end{remark}

\subsection{Anisotropic wave propagation in two dimensions}
The nematic field introduces an anisotropy in the wave propagation, as already shown in \Cref{fig:intro:waveguide}. 
In this example, we demonstrate this phenomenon in a scattering scenario. Inside a rectangle $D_0 = (0,1.35)^2$, we place an obstacle $D_1$ defined as the interior of a polygon with vertices 
\begin{align*}
   \{ (0.5,     0.4),
   (0.5475,  0.3725),
   (0.5285,  0.4),
   (0.5475,  0.4275) \}.
\end{align*}
Then, we set $\Omega \coloneqq D_0 \setminus D_1 \subset \mathbb{R}^2$ with $\GammaD = \partial D_1$ and $\GammaR = \partial D_0$.
For the source term, we choose a Gaussian pulse centered at $(0.25,0.5)$, that is,
\begin{align*}
   f(\bm{x}) = \sqrt{500/\pi} \exp \left( -500((x_1-0.25)^2 + (x_2-0.5)^2) \right), \quad \bm{x} \in \mathbb{R}^2.  
\end{align*}

In \Cref{fig:numex:NematicScattering}, we compare the real part of the computed solutions for three nematic fields. As a baseline case without anisotropy, we consider $\bm{n} = (0,0)^{\top}$. We then compare configurations in which the nematic field is aligned with the two "wings" of the obstacle, that is, the two sides of the obstacle facing the source. Since the wave propagates faster in the direction of the nematic field, this configuration accelerates the wave propagation on one side of the obstacle, while slowing it down on the other.
Indeed, we observe that the resulting scattering patterns differ significantly.

\begin{figure}[!htbp]
   \centering
   \includegraphics[width=0.98\textwidth]{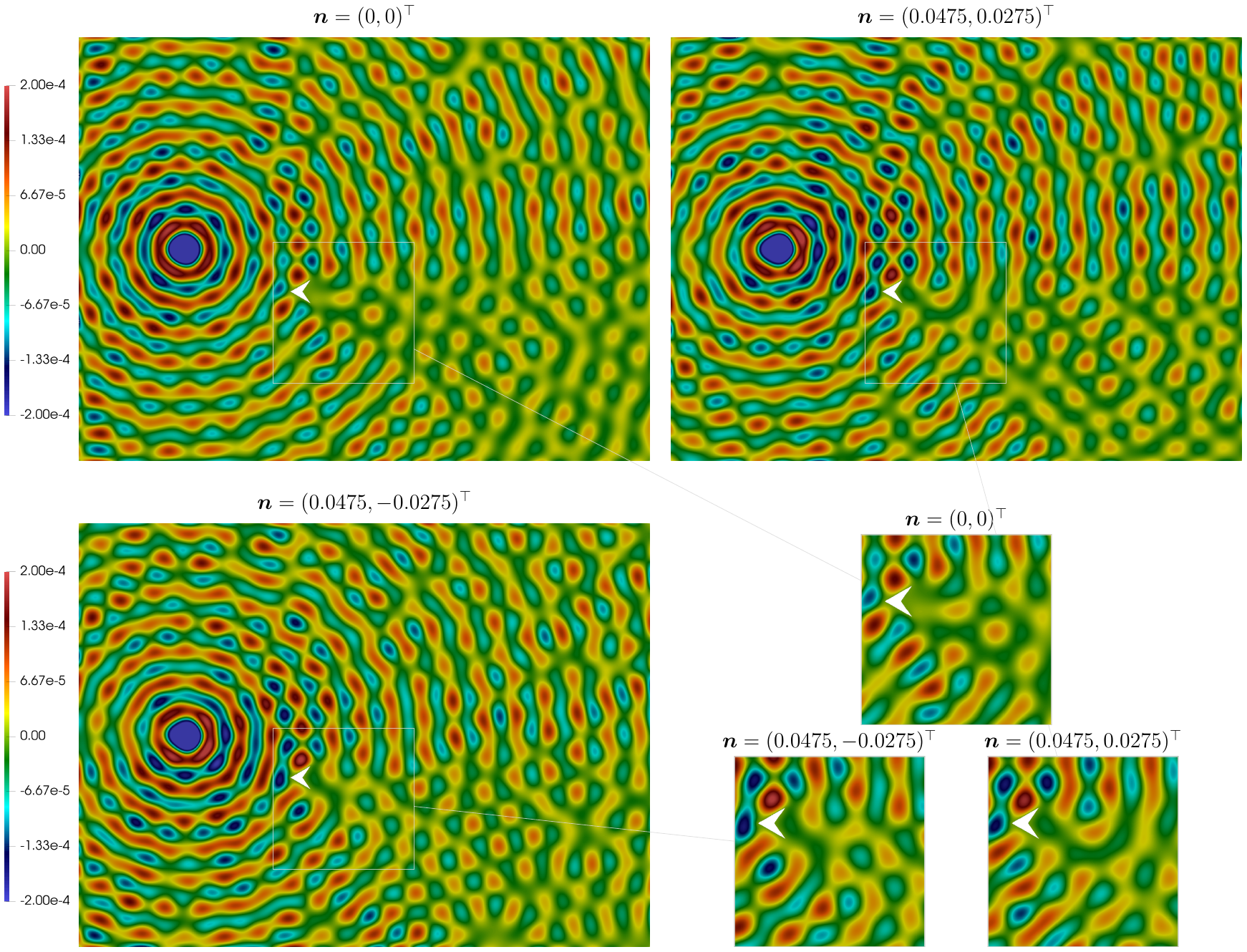}
   \caption{We consider the scattering of a Gaussian pulse by a sound-soft obstacle with the parameters $\alpha = 5 \cdot 10^{-4}$, $\beta = 10^{-4}$ and $k = \theta = 200$. We compute with polynomial degree $p = 9$ on a mesh with mesh size $h \approx 0.02$ (giving $7686$ elements). For the nematic fields $\bm{n} = (0,0)^{\top}$, $\bm{n} = (0.0475,0.0275)^{\top}$, and $\bm{n} = (0.0475,-0.0275)^\top$, we compare the real part of the computed solution and highlight the difference in the respective scattering patterns.}
   \label{fig:numex:NematicScattering}
\end{figure}

\subsection{Anisotropic wave propagation in three dimensions}
The flexibility of the $C^0$-HIP discretization allows us to consider three-dimensional problems.
The goal of this section is to showcase the anisotropic effect of the nematic field already observed in the previous section in the three-dimensional setting. 
We choose $\Omega$ to be a ball of radius $0.85$ centered at the origin and cut out two balls of radius $0.2$ centered at $(0,0.4,0)$ and $(0,-0.1,0)$. Then, we set impedance boundary conditions at the exterior boundary and sound-soft boundary conditions at the interior and consider a Gaussian pulse as a right-hand side
\begin{align*}
   f(\bm{x}) := \sqrt{10^3/\pi} \exp \left( -10^3((x_1-0.4)^2 + x_2^2 + x_3^2) \right).
\end{align*}

\begin{figure}[!htbp]
   \begin{center}
     \includegraphics[width=0.98\textwidth]{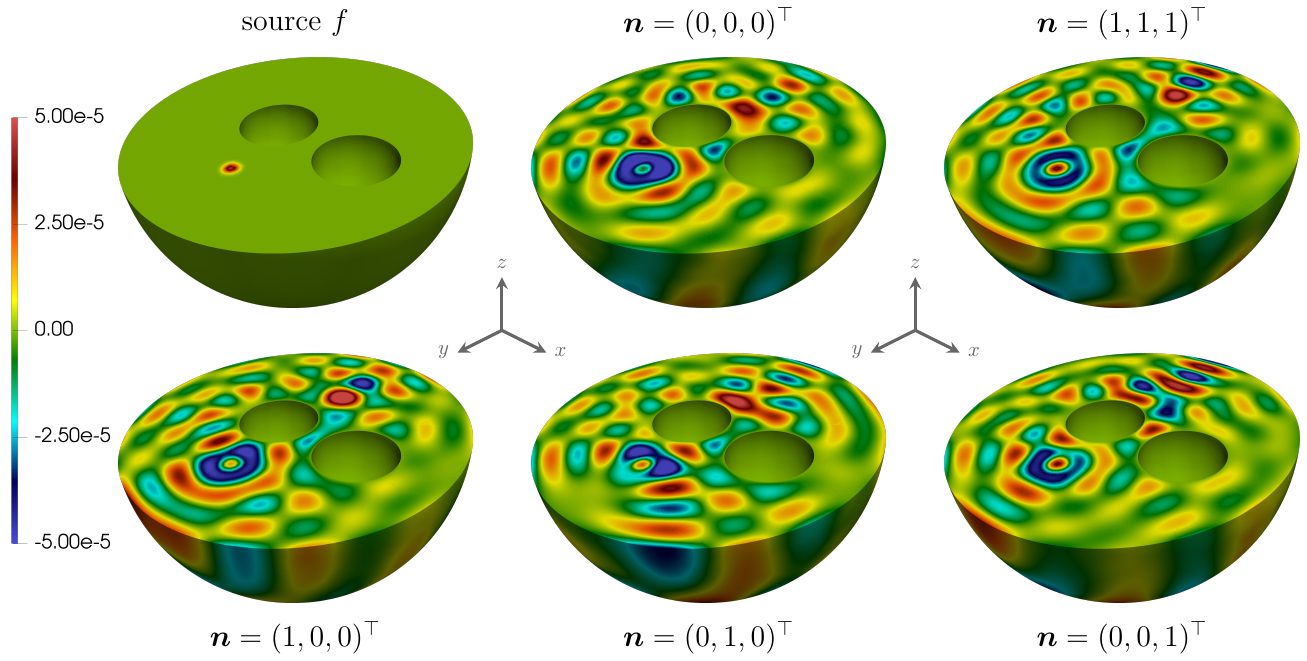}
   \end{center}
   \caption{We compare the real part of the computed solution with parameters $\alpha = 3 \cdot 10^{-2}$, $\beta = 10^{-2}$, and $k = \theta = 135$ for five different nematic fields. The solutions are computed with polynomial degree $p =5$ on a curved mesh with roughly 9500 elements (mesh size $h = 0.1$). The colorbar only applies to the computed solutions, not the picture of the source term.}
   \label{fig:numex:3dScattering}
\end{figure}

\Cref{fig:numex:3dScattering} displays the computed solution with five different nematic fields. The different nematic directors lead to drastically different scattering patterns. For $\bm{n}$ being the three unit vectors, we highlight this in \Cref{fig:numex:3DEvaluation}.

\begin{figure}[!htbp]
   \centering
   \includegraphics[width=0.98\textwidth]{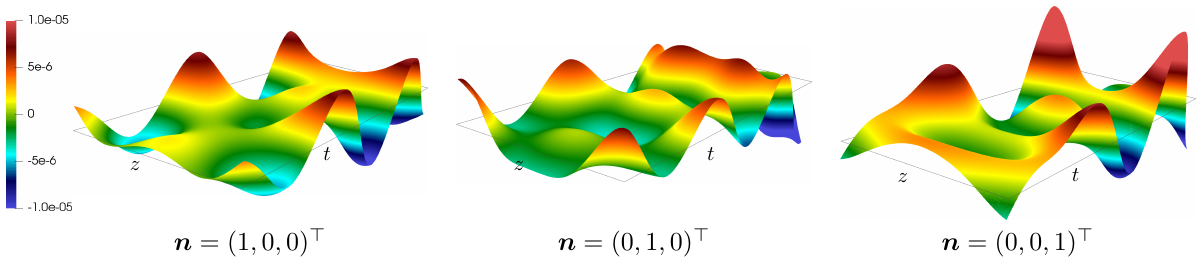}
   \caption{We highlight the difference in the scattering patterns for different nematic fields. For the cases where $\bm{n}$ is one of the three unit vectors, we show evaluations of the solutions measured at the cross-section $(0.7 \sin(t),0.7 \cos(t)) \times [-0.2,0.2]$, $t \in [-\pi/4, \pi/4]$.}
   \label{fig:numex:3DEvaluation}
\end{figure}

\section*{Acknowledgements}
The author would like to thank Christoph Lehrenfeld and Umberto Zerbinati for helpful discussions.

\appendix
\section{Details on \Cref{thm:discreteGarding}}\label{appendix:Tcomp}
For completeness, we state the result from \cite[Thm.~3]{HLS22H1} that implies \Cref{thm:discreteGarding}. In addition to the setting and definitions from \Cref{sec:prelim}, we say that a bounded sesquilinear form $a(\cdot,\cdot) :  V \times V \to \mathbb{C}$ is \emph{T-coercive} \cite{C25Tcoerc}, if there exists a bijection $T \in \mathcal{L}(V)$ such that $a(T \cdot,\cdot)$ is coercive. We call a compact perturbation of a T-coercive problem \emph{weakly T-coercive}. 

The following theorem shows the stability of the discrete problem if the weakly T-coercive structure can be transferred to the discrete level in a suitable manner. We state the result for the associated (sequences of) bounded linear operators and note that the notion of asymptotic consistency from \Cref{def:asympConsistency} reads in operator form as 
\begin{align}
   \lim_{h \to 0} \Vert A_h p_h u - p_h A u \Vert_{V_h} = 0 \qquad \forall u \in V.
\end{align}
Further, we require the following notion of compactness for sequences of operators defined on the finite dimensional spaces $V_h$. 

\begin{definition}[Compactness]\label[definition]{def:SeqCompactness}
   We call a sequence of operators $\seqh{K}$, $K_h \in \mathcal{L}(V_h)$ \emph{compact} if for every uniformly bounded sequence $\seqh{v}$, $v_h \in V_h$, the sequence $\seqh{K_h v}$ has a convergent subsequence.
\end{definition}

\begin{theorem}\label{thm:appendix:compatibility}
   Let $A \in \mathcal{L}(V)$ be bijective and $\seqh{A}$, $A_h \in \mathcal{L}(V_h)$, be asymptotically consistent with $A$. 
   Suppose there exists $\seqh{A}$, $\seqh{T}$, $\seqh{B}$, $\seqh{K}$ and operators $B, T \in \mathcal{L}(V)$ such that 
   \begin{enumerate}[label=(\roman*)]
      \item we have that $A_h T_h = B_h + K_h$, 
      \item $\seqh{T}$ is stable and $\seqh{T}$ is asymptotically consistent with $T$, 
      \item $\seqh{B}$ is stable, $B \in \mathcal{L}(V)$ is bijective and $\seqh{B}$ is asymptotically consistent with $B$,
      \item $\seqh{K}$ is compact, 
   \end{enumerate}
   Then $\seqh{A}$ is asymptotically stable. 
\end{theorem}

\begin{proof}
   We refer to \cite[Thm.~3]{HLS22H1} and \cite[Lem.~1]{HLS22H1}.
\end{proof}

By setting $T = \operatorname{id}_V$ and $T_h = \operatorname{id}_{V_h}$, we can apply the result to non-conforming approximations of compact perturbations of coercive operators. To show that we can apply this to the setting from \Cref{lem:GardingsInequality} and \Cref{thm:discreteGarding}, we first show that \Cref{ass:InclusionConsistentlyCompact} implies the compactness of the sequence of operators associated with the perturbation in \eqref{eq:Abstract:DiscrGarding}.

\begin{lemma}\label[lemma]{lem:appendix:AssumpImpliesCompactness}
   Let \Cref{ass:InclusionConsistentlyCompact} be satisfied. Then the sequence of operators $\seqh{K}$, $K_h \in \mathcal{L}(V_h)$, defined by 
   \begin{align*}
      (K_h u_h,v_h)_{V_h} \coloneqq (\iota_h u_h, \iota_h v_h)_{W} \qquad \forall u_h,v_h \in V_h, 
   \end{align*}
   is compact in the sense of \Cref{def:SeqCompactness}.
\end{lemma}

\begin{proof}
   Let $\seqh{v}$, $v_h \in V_h$, be uniformly bounded. Passing to a subsequence, there exists $v \in V$ such that $v_h \overset{\pi_h}{\rightharpoonup} v$; due to \Cref{ass:InclusionConsistentlyCompact}, there then exists $J \in \mathcal{L}(V,W)$ such that $\iota_h v_h \to Jv$ strongly in $W$. Let $w \in V$ be defined by $(w,\phi)_{V} = (Jv,J\phi)_{W}$ for all $\phi \in V$. Then, we have that
   \begin{align*}
      (K_h v_h,\pi_h \phi)_{V_h}  = (\iota_h v_h,\iota_h(\pi_h \phi))_{W} \to (Jv,J \phi)_{W} = (w,\phi)_{V} \qquad \forall \phi \in V, 
   \end{align*}
   since $\pi_h \phi \overset{\pi_h}{\to} \phi$. In other words, $\seqh{K_h v} \overset{\pi_h}{\rightharpoonup} w$. By \Cref{ass:InclusionConsistentlyCompact}, this implies that $\iota_h K_h v_h \to Jw$ strongly in $W$ and thus 
   \begin{align*}
      \Vert K_h v_h \Vert^2_{V_h} = (\iota_h v_h, \iota_h K_h v_h)_{W} \to (Jv,Jw)_{W} = (w,w)_{V} = \Vert w \Vert^2_{V}.
   \end{align*}
   Therefore, we obtain the compactness of $\seqh{K}$ in the sense of \Cref{def:SeqCompactness}, since
   \begin{align*}
      \Vert K_h v_h - \pi_h w \Vert^2_{V_h} = \Vert K_h v_h \Vert^2_{V_h} - 2 \Re (K_h v_h, \pi_h w)_{V_h} + \Vert \pi_h w \Vert^2_{V_h} \to 0.
   \end{align*}
\end{proof}

This allows us to prove \Cref{thm:discreteGarding}. 

\begin{proof}[Proof of \Cref{thm:discreteGarding}]
   By assumption, the continuous operator $A \in \mathcal{L}(V)$ can be written as $A = B + K$, where $B \in \mathcal{L}(V)$ is coercive and $K \in \mathcal{L}(V)$ is compact, and the sequence of discrete operators $\seqh{A}$ may be written as $A_h = B_h + K_h$, where $\seqh{B}$ is uniformly coercive due to \eqref{eq:Abstract:DiscrGarding} and $\seqh{K}$ is compact in the sense of \Cref{def:SeqCompactness} by \Cref{lem:appendix:AssumpImpliesCompactness}. Replacing the constants in front of $\Vert \cdot \Vert^2_{W}$ in \eqref{eq:Abstract:Garding} and \eqref{eq:Abstract:DiscrGarding} by $\max\{C_W,\tilde{C}_W\}$, we obtain that $K_h$ is asymptotically consistent with $K$ and that $B_h$ is asymptotically consistent with $B$. Thus, \Cref{thm:appendix:compatibility} yields the asymptotic stability of $\seqh{A}$.
\end{proof}

\bibliography{main}

\end{document}